\documentclass[a4paper]{amsart}

\usepackage{amscd,amsthm,amsfonts,amssymb,amsmath,esint}
\usepackage[colorlinks=true]{hyperref}
\usepackage{fullpage}
\usepackage[all]{xy}

\usepackage{mathrsfs}
\usepackage{enumerate}
\usepackage{mathtools}

     \newcommand{\BF}{{\mathbb {F}}}

    \newcommand{\BQ}{{\mathbb {Q}}}

     \newcommand{\BZ}{{\mathbb {Z}}}

    \newcommand{\CO}{{\mathcal {O}}}

     \newcommand{\fH}{{\mathfrak{H}}}
     
    \newcommand{\fK}{{\mathfrak{K}}} 
    \newcommand{\fM}{{\mathfrak{M}}}

    \newcommand{\fS}{{\mathfrak{S}}} 
     
     \newcommand{\fX}{{\mathfrak{X}}}

    \newcommand{\corank}{{\mathrm{corank}}}

    \newcommand{\coker}{\mathrm{coker}~\!}\newcommand{\cyc}{{\mathrm{cyc}}}

    \newcommand{\Gal}{{\mathrm{Gal}}} \newcommand{\GL}{{\mathrm{GL}}}
    
    \newcommand{\Hom}{{\mathrm{Hom}}}

    \newcommand{\ord}{{\mathrm{ord}}} \newcommand{\rank}{{\mathrm{rank}}}

    \newcommand{\Sel}{{\mathrm{Sel}}}

    \newcommand{\Ext}{{\mathrm{Ext}}}

    \font\cyr=wncyr10

    \newcommand{\Sha}{\hbox{\cyr X}}

    \theoremstyle{plain}
    \newtheorem{thm}{Theorem}[section] \newtheorem{cor}[thm]{Corollary}
    \newtheorem{lem}[thm]{Lemma}  \newtheorem{prop}[thm]{Proposition}
    \newtheorem {conj}[thm]{Conjecture} \newtheorem{defn}[thm]{Definition}

\theoremstyle{remark} \newtheorem{remark}[thm]{Remark}
\theoremstyle{remark} 
\theoremstyle{remark}

    \numberwithin{equation}{section}

\begin{document}

\title{Non-commutative Iwasawa theory of abelian varieties over global function fields}

\author[Li-Tong Deng]{Li-Tong Deng}
\address{\textit{Li-Tong Deng}, Qiuzhen College, Tsinghua University, 100084, Beijing, China}
\email{\it dlt23@mails.tsinghua.edu.cn}

\author[Yukako Kezuka]{Yukako Kezuka}
\address{\textit{Yukako Kezuka}, Kanazawa University, Natural Science and Technology, Kakumamachi, Kanazawa, 920-1164, Ishikawa, Japan.}
\email{\it kezuka@se.kanazawa-u.ac.jp}

\author[Yong-Xiong Li]{Yong-Xiong Li}
\address{\textit{Yong-Xiong Li}, Yanqi Lake Beijing Institute of Mathematical Sciences and Applications,
 No. 544 Hefangkou
Village, Huaibei Town, Huairou District, 101408, Beijing, China}
\email{\it liyongxiong@bimsa.cn}

\author[Meng Fai Lim]{Meng Fai Lim}
\address{\textit{Meng Fai Lim}, School of Mathematics and Statistics, Hubei Key Laboratory of Mathematical Sciences, Central China Normal University, 430079, Wuhan, China}
\email{\it limmf@ccnu.edu.cn}


\begin{abstract}
Let $A$ be an abelian variety defined over a global function field $F$, and let $p$ be a prime distinct from the characteristic of $F$. Let $F_\infty$ be a $p$-adic Lie extension of $F$ that contains the cyclotomic $\mathbb{Z}_p$-extension $F^{\mathrm{cyc}}$ of $F$. In this paper,
we investigate the structure of the $p$-primary Selmer group $\mathrm{Sel}(A/F_\infty)$ of $A$ over $F_\infty$. We prove the $\mathfrak{M}_H(G)$-conjecture for $A/F_\infty$.
Furthermore, we show that both the $\mu$-invariant of the Pontryagin dual of the Selmer group $\mathrm{Sel}(A/F^\cyc)$ and the generalised $\mu$-invariant of the Pontryagin dual of the Selmer group $\mathrm{Sel}(A/F_\infty)$ are zero, therby proving Mazur's conjecture for $A/F$. We then relate the order of vanishing of the characteristic elements, evaluated at Artin representations, to the corank of the Selmer group of the corresponding twist of $A$ over the base field $F$. Assuming the finiteness of the Tate--Shafarevich group, we establish  that this corank equals the order of vanishing  of the $L$-function of $A/F$ at $s=1$.
Finally, we extend a theorem of Sechi---originally proved for elliptic curves without complex multiplication---to abelian varieties over global function fields. This is achieved by adapting the notion of generalised Euler characteristic, introduced by Zerbes for elliptic curves over number fields. This new invariant allows us, via Akashi series, to relate the generalised Euler characteristic of $\mathrm{Sel}(A/F_\infty)$ to the Euler characteristic of $\mathrm{Sel}(A/F^{\mathrm{cyc}})$.
\end{abstract}

\maketitle

\section{Introduction}\label{section0}

Non-commutative Iwasawa theory has emerged as a powerful framework for understanding deep arithmetic properties over number fields contained in a $p$-adic Lie extension and their precise relationship with special values of complex $L$-functions. The aim of the present paper is to extend the study of non-commutative Iwasawa theory to the setting of global function fields. To motivate our discussion, we begin by reviewing the number field context as developed by Coates, Fukaya, Kato, Sujatha, and Venjakob in \cite{c-non} (see also \cite{burns, BV, ven}).  In the number field setting, let $p$ denote an odd prime.  Given a compact $p$-adic Lie group $\mathcal{G}$, the Iwasawa algebra $\Lambda(\mathcal{G})$ of $\mathcal{G}$ is defined as the completed group ring
\[\Lambda(\mathcal{G})=\varprojlim \BZ_p[\mathcal{G}/\mathcal{U}],\]
where the inverse limit is taken over all open normal subgroups $\mathcal{U}$ of $\mathcal{G}$. Let $E$ be an elliptic curve defined over a number field $F$, with good ordinary reduction at all primes of $F$ lying above $p$. Let $F_\infty$ be a $p$-adic Lie extension of $F$ such that the cyclotomic $\BZ_p$-extension $F^{\mathrm{cyc}}$ of $F$ is contained in $F_\infty$, and the Galois group $G = \Gal(F_\infty/F)$ is a $p$-adic Lie group of positive dimension. Write $\Gamma=\Gal(F^{\mathrm{cyc}}/F)$, $H=\Gal(F_\infty/F^{\mathrm{cyc}})$, and define
\[X(E/F_\infty):=\Hom_{\rm cts}(\Sel(E/F_\infty),\BQ_p/\BZ_p),\]
the Pontryagin dual of the $p$-primary Selmer group $\Sel(E/F_\infty)$, endowed with its natural left $\Lambda(G)$-module structure.  In the case where $F_\infty=F^{\mathrm{cyc}}$, it is well known that $X(E/F^{\mathrm{cyc}})$ is a finitely generated $\Lambda(\Gamma)$-module. It is generally expected that $X(E/F^{\mathrm{cyc}})$ enjoys favorable module-theoretic properties, as reflected in the following conjecture of Mazur \cite{mazur}:

\begin{conj}[Mazur] The Iwasawa module $X(E/F^{\mathrm{cyc}})$ is a finitely generated torsion $\Lambda(\Gamma)$-module.
\end{conj}

Mazur verified this conjecture in the case where $\Sel(E/F)$ is finite. The strongest evidence supporting this conjecture is a deep result of Kato \cite{kato}, who established the torsion property when $E$ is defined over $\BQ$ and $F$ is an abelian extension of $\BQ$. Taking into account the structure theory of $\Lambda(\Gamma)$-modules, Mazur's conjecture is equivalent to asserting that the quotient $X(E/F^\cyc)/ X(E/F^\cyc)(p)$ is finitely generated over $\BZ_p$, where $X(E/F^\cyc)(p)$ denotes the $\BZ_p$-torsion submodule of $X(E/F^\cyc)$. Motivated by this observation, Coates and collaborators introduced the category $\fM_H(G)$, consisting of all finitely generated $\Lambda(G)$-modules $M$ such that the quotient $M/M(p)$ by its $\BZ_p$-torsion submodule $M(p)$ is a finitely generated $\Lambda(H)$-module. They then conjectured that this category contains all the torsion $\Lambda(G)$-modules of arithmetic significance. More precisely, they formulated (see Section 5 and Conjecture 5.1 of \cite{c-non}):

\begin{conj}[$\fM_H(G)$-Conjecture] The Iwasawa module $X(E/F_\infty)$ belongs to the category $\fM_H(G)$.
\end{conj}

This conjecture is crucial, as it enables the definition of a characteristic element for $X(E/F_\infty)$, which is essential for formulating the main conjectures in non-commutative Iwasawa theory (see \cite{c-non, BV}). Coates \mbox{\textit{et al.}\,}further demonstrated how these characteristic elements relate to the Euler characteristics of $X(E/F_\infty)$ via the theory of Akashi series introduced in \cite{coates-sujatha-schneider}. This connection has been further explored by Zerbes in a series of papers \cite{zerbes1, zerbes2, zerbes3}, where she related the characteristic element to the generalised Euler characteristic of $X(E/F_\infty)$. The notion of the generalised Euler characteristic is an extension of the classical Euler characteristic. The idea first appeared in \cite{coates-sujatha-schneider} and was later developed and systematically studied by Zerbes in \cite{zerbes2}.

We now shift our focus to the function field setting, which forms the central theme of this paper. In view of the well-known analogy between number fields and global function fields, one naturally expects that Iwasawa-theoretic phenomena should also manifest in the function field setting.  Indeed, several works have explored this direction (see \cite{BV, ochitri, Ray, sechi, TV, v1, Wi}).
From now on, let $p$ denote a fixed prime, where we allow $p=2$. Throughout this paper, $F$ will denote a global function field, that is, a finite extension of the field $k(T)$ of rational functions in one variable over a finite field $k$. We always assume that the characteristic $\mathrm{char}(F)$ of $F$ is distinct from our fixed prime $p$. Let $\bar{F}$ be a separable closure of $F$. The cyclotomic $\BZ_p$-extension $F^{\rm cyc}$ of $F$ is obtained by adjoining to $F$ the unique $\BZ_p$-constant field extension of $k$. Writing $\mu_{p^\infty}$ for the group of $p$-power roots of unity in an algebraic closure $\bar{k}$ of $k$, the field $F^{\rm cyc}$  is the unique $\BZ_p$-extension of $F$ contained in $F(\mu_{p^\infty})$. Following the analogy with the number field setting, we define an \textit{admissible} $p$-adic Lie extension $F_\infty/F$ as a Galois extension satisfying the following properties:
\medskip
\begin{enumerate}\label{adm-ext}
\item  $F_\infty$ contains the cyclotomic $\BZ_p$-extension $F^\mathrm{cyc}$ of $F$.
\item $G = \Gal(F_\infty/F)$ is a compact $p$-adic group without $p$-torsion.
\item The extension $F_\infty/F$ is unramified outside a finite set of primes.
\end{enumerate}

It is worth noting that a global function field $F$ with $\mathrm{char}(F)\neq p$ admits only one $\BZ_p$-extension, namely the cyclotomic one (see \cite[Proposition 10.3.20]{nsw}). Thus, any admissible extension of dimension at least two is necessarily non-commutative, highlighting the relevance of non-commutative Iwasawa theory in the function field context. In contrast, when $\mathrm{char}(F) = p$, there are infinitely many geometric (i.e., not containing any constant field extension) $\BZ_p$-extensions of $F$ (see \cite[Theorem 3]{GK}), which marks a significant distinction between the number field and function field cases.

We now return to the main results of this paper.  Our first main result is as follows: 

\begin{thm} \label{introMHG} Let $p$ be any prime. Let $A$ be an abelian variety over a global function field $F$ with characteristic prime to~$p$, and let $F_\infty/F$ be an admissible $p$-adic Lie extension. Then the $\mathfrak{M}_H(G)$-conjecture holds for the Pontryagin dual $X(A/F_\infty)$ of the Selmer group $\Sel(A/F_\infty)$.
Furthermore,  the Pontryagin dual $X(A/F^\cyc)$ of the Selmer group $\Sel(A/F^\cyc)$ has $\mu$-invariant equal to zero, and the generalised $\mu$-invariant of $X(A/F_\infty)$ is zero. In particular, Mazur's conjecture holds for $A/F$.
\end{thm}

\bigskip

We now make several remarks. This theorem completely resolves Mazur's conjecture and the $\mathfrak{M}_H(G)$-conjecture in the setting of global function fields. This presents a stark contrast to the number field case.  It has long been known that the Pontryagin dual of the Selmer group over the cyclotomic $\BZ_p$-extension of a number field can have a positive $\mu$-invariant (see \cite[\S 10, Example 2]{mazur}, \cite{drinen}). By contrast, the theorem shows that such a phenomenon never occurs in the function field setting when $\mathrm{char}(F)\neq p$. We also note that Sechi \cite{sechi} previously proved this result for elliptic curves $E$ without complex multiplication, defined over a global function field $F$, where $F_\infty$ is obtained by adjoining the $p$-power torsion points of $E$ to $F$, under the assumption that $\Sel(E/F)$ is finite.

In the case where $\mathrm{char}(F)=p$, we similarly define $F^{\rm cyc}$ to be the unique constant field $\BZ_p$-extension of $F$.
Let $F_\infty/F$ be a $p$-adic Lie extension satisfying conditions analogous to (1), (2) and (3) described above. For an abelian variety $A/F$, we assume that all primes of bad reduction are ordinary and lie among the primes that ramify in $F_\infty/F$. Under the additional assumption that the $\mu$-invariant of $\Sel(A/F^{\rm cyc})$ vanishes, Ochiai and Trihan proved in \cite[Theorem 1.9]{ochitri} that the $\mathfrak{M}_H(G)$-conjecture holds in this setting.
However, recent work by Lai et al.\ \cite{LLSTT} suggests that $\mu$ may be positive when $\mathrm{char}(F) = p$, thereby highlighting a contrast with the $\mathrm{char}(F) \ne p$ case, which is the main focus of this paper.

Our proof of Theorem \ref{introMHG} differs entirely from previous works on the $\fM_H(G)$-conjecture. A key element of our proof is the observation that the cohomological dimension of $\Gal(\bar{F}/F^\cyc)$ is equal to one. From this starting point, our method proceeds via a purely cohomological approach. (Upon completing our work, the authors learned that Witte had proven a similar result on the $\mathfrak{M}_H(G)$-conjecture using entirely different methods (see \cite{Wi}); we thank Jishnu Ray and Fabien Trihan for pointing this out.)
While Witte's result on the $\mathfrak{M}_H(G)$-conjecture may initially appear more general, as his approach involves the use of Selmer complexes, our result recovers his, given the well-documented relationship between Selmer groups and Selmer complexes (see \cite[Propositions 4.2.35 and 4.3.13]{FK}; see also \cite[Lemma 3.3.2]{LimMHGcong}). Of course, this same relationship implies that our result follows from that of Witte. That said, we would like to elaborate on the motivation behind the alternative approach adopted in our paper. First of all, the focus of our work differs significantly from that of Witte's. While Witte's paper is centered on proving a main conjecture, our study is dedicated to investigating the precise structure of Selmer groups and the order of vanishing of their characteristic elements. More precisely, our direct approach proves that the global and local cohomology groups in the defining sequence of the Selmer group are cofinitely generated over $\Lambda(H)$ (see Lemma 3.3 and the discussion preceding Lemma 5.3), a fact which does not appear to be explicitly established in Witte's work. In contrast, within the number field context, these cohomology groups are not cofinitely generated over $\Lambda(H)$ (see \cite[Theorems 3.2 and 4.1]{OV}), necessitating an elaborate snake lemma argument for the computation of the Akashi series of a Selmer group (as in \cite{coates-sujatha-schneider, zerbes2, zerbes3}). In the function field context, our observation that these groups are $\Lambda(H)$-cofinitely generated allows for a more straightforward computation of the Akashi series of Selmer groups by directly computing the Akashi series of the global and local cohomology groups (see Section 5 for details). Additionally, we show that the generalised $\mu$-invariant of the dual Selmer group $X(A/F_\infty)$, as well as the $\mu$-invariant of the dual Selmer group $X(A/F^\cyc)$, are both equal to zero (see Corollary \ref{genmu}). We note that the vanishing of the $\mu$-invariant is already implicit in Witte~\cite[Corollary~4.21]{Wi}, where it is shown that the dual Selmer group lies in a certain category $\mathbf{N}_H$, which essentially also implies that $\mu=0$. Finally, without relying on a non-commutative main conjecture, we present results related to the Birch--Swinnerton-Dyer conjecture, based on Tate's fundamental work, under an assumption on the finiteness of the Tate--Shafarevich group (see Theorem \ref{thm2intro} and the subsequent discussion).

In light of Theorem \ref{introMHG}, we can now attach a characteristic element $\xi_A$ to $X(A/F_\infty)$ in the sense of Coates \textit{et al.} This enables us to evaluate $\xi_A$ at a given Artin representation $\rho: G\to \GL_n(\mathcal{O})$, where $\mathcal{O}$ is the ring of integers of a finite extension of $\BQ_p$ in $\overline{\BQ}_p$, via the homomorphism $\Phi_\rho$  (for details, see Section \ref{section7}, particularly \eqref{Phi_rho}). This yields an element in the field of fractions of the formal power series ring $\mathcal{O}[[T]]$ in an indeterminate $T$ with coefficients in $\mathcal{O}$.
To facilitate explicit calculations in the discussion that follows, we introduce the following assumption:
\[\mathbf{(G)}: \quad H^i(G, A_{p^\infty}(F_\infty)) \text{ is finite for all } i \geq 1.\]

We are now in a position to state our next result, presented here in a slightly simplified form. For a more precise andgeneral version, we refer readers to Theorem \ref{ordreg}, Corollary \ref{corbsd}, and Theorem \ref{ordrho}.

\begin{thm}\label{thm2intro} Let $A$ be an abelian variety defined over a global function field $F$ with $\mathrm{char}(F)\neq p$, and let $F_\infty/F$ be an admissible $p$-adic Lie extension. Denote by $\xi_A$ the characteristic element of $X(A/F_\infty)$, and let $\mathrm{reg}_F$ denote the regular representation of $F$. Assuming $(\mathbf{G})$ holds, we have
$$\ord_{T=0}(\Phi_{\mathrm{reg}_F}(\xi_A))\geq \corank_{\BZ_p}(\Sel(A/F)).$$
Moreover, equality holds if Greenberg's semi-simplicity conjecture (Conjecture \ref{greenberg}) holds for $A/F$. Additionally, if we assume the finiteness of $\Sha(A/F)$, then
\[ \ord_{T=0}(\Phi_{\mathrm{reg}_F}(\xi_A))= \ord_{s=1}L(A/F,s). \]
\end{thm}

\bigskip

In fact, by the deep theorem of Kato--Trihan \cite{kato-trihan}, the finiteness of the $\ell$-primary part of the Tate-Shafarevich group $\Sha(A/F)(\ell)$ for a single prime number $\ell$ is sufficient to establish the final equality in the theorem.

\bigskip

To state our final result, we first introduce the following stronger finiteness assumption. Recall that we write $H=\Gal(F_\infty/F^{\rm cyc})$.
\[\mathbf{(H)}: \quad H^i(H, A_{p^\infty}(F_\infty)) \text{ is finite for all } i \geq 0.\]

We provide a criterion for this condition to hold in Proposition \ref{propH} of the paper (see also \cite[Section 5]{zerbes2} and \cite[Remark 6.6]{LimMHGcong} for cases where ($\mathbf{H}$) is known to hold in the number field context). Let $S$ be a nonempty finite set of primes of $F$ containing all the primes where the abelian variety $A$ has bad reduction, as well as all the primes of $F$ that ramify in $F_\infty$. Let $S'$ denote the subset of primes in $S$ whose inertia group in $G$ is infinite. We write $A^*$ for the dual abelian variety. The following is our final main result.

\begin{thm}\label{thm3intro}
Let $A$ be an abelian variety defined over a global field $F$ with $\mathrm{char} F\neq p$ and let $F_\infty$ be an admissible $p$-adic Lie extension of $F$. Assume that $(\mathbf{H})$ holds. Then $\Sel(A/F_\infty)$ has a finite generalised $G$-Euler characteristic $\chi(G, \Sel(A/F_\infty))$ if and only if $\Sel(A/F^{\rm cyc})$ has a finite generalised $\Gamma$-Euler characteristic $\chi(\Gamma, \Sel(A/F^\cyc))$. Moreover, if this is the case, then
\[\chi(G, \Sel(A/F_\infty)) = \chi(\Gamma, \Sel(A/F^\cyc)) \prod_{v\in S'}\frac{\#  A^*_{p^\infty}(F_v)}{\# H^1(\Gamma_w, A_{p^\infty}(F^\cyc_w))},\]
where, for each $v\in S'$, $w$ is a fixed prime of $F^\cyc$ lying above $v$, and $\Gamma_w$ denotes the decomposition subgroup of $\Gamma$ at $w$.
\end{thm}

\medskip

This result is a natural analogue of Zerbes \cite{zerbes2}, who proved a similar statement in the context of number fields. In the function field context, our theorem generalises earlier computations by Sechi \cite{sechi} and Valentino \cite{v1}. When $A$ is an elliptic curve, our formula recovers Sechi's result, in which Euler factors of the complex $L$-function appear on the right-hand side. For explicit examples, we refer the reader to \cite{sechi}.

\bigskip

We conclude this introduction with an outline of the paper.  Section~\ref{section2} collects foundational material on Iwasawa algebras, which will be used throughout the paper. The proof of Theorem \ref{introMHG} is given in Section~\ref{section3}. The key observation is that the Galois group $\Gal(F_S/F^{\mathrm{cyc}})$, where $F_S$ is the maximal algebraic extension of $F$ contained in $\bar{F}$ which is unramified outside $S$, has $p$-cohomological dimension one (see Lemma \ref{cdpcyc}). In Section~\ref{section4}, we establish the surjectivity of the localisation map defining the Selmer group $\Sel(A/F_\infty)$, using Jannsen's spectral sequence and Nekov\'{a}\v{r}'s duality theorem. This is a key input in our computation of Akashi series in Section~\ref{section5}. In Section~\ref{section6}, we prove a control theorem, showing in particular that the restriction map $\Sel(A/F)\to \Sel(A/F^{\mathrm{cyc}})^{\Gamma}$ has a finite kernel and cokernel. This allows us to relate the order of vanishing at $T=0$ of the characteristic power series of $\Sel(A/F^{\mathrm{cyc}})$ to the $\BZ_p$-corank of the Selmer group $\Sel(A/F)$.   Section~\ref{section7} relates the order of vanishing of characteristic elements evaluated at Artin representations to  Selmer coranks and their twists over intermediate subextensions of the $p$-adic Lie extension. This yields Theorem \ref{thm2intro}. In Section~\ref{section8}, we introduce the definition of the generalised Euler characteristic, following the framework of Zerbes \cite{zerbes2}. Finally, in Section~\ref{section9}, we compute the generalised Euler characteristics under the assumption $(\mathbf{H})$, thereby establishing Theorem \ref{thm3intro}.

\section{Prerequisites on Iwasawa algebras}\label{section2}

Throughout, let $p$ denote a fixed prime. Let $G$ be a compact $p$-adic Lie group without $p$-torsion. The Iwasawa algebra of $G$ over $\BZ_p$ is defined by
\[ \Lambda(G)= \varprojlim \BZ_p[G/U], \]
where $U$ runs over the open normal subgroups of $G$, and the inverse limit is taken with respect to the natural projection maps. Assume for now that $G$ is a pro-$p$ group. Then the ring $\Lambda(G)$ admits a skew field of fractions $\mathcal{Q}(G)$, which allows us to define the rank of a $\Lambda(G)$-module $M$ as follows:
\[ \rank_{\Lambda(G)}(M) = \dim_{\mathcal{Q}(G)}\mathcal{Q}(G)\otimes_{\Lambda(G)}M. \]
The module $M$ is said to be torsion if $\rank_{\Lambda(G)}(M)=0$.

Now, suppose that $G$ is not necessarily pro-$p$. By Lazard's theorem (see \cite[Corollary 8.34]{uni-gp}),
one can always find an open subgroup $G'$ of $G$ that is pro-$p$. In this case, a $\Lambda(G)$-module $M$ is said to be torsion over $\Lambda(G)$ if it is torsion over $\Lambda(G')$ in the sense defined above. We shall also make use of the following equivalent characterisation of a torsion $\Lambda(G)$-module: $\Hom_{\Lambda(G)}(M,\Lambda(G))=0$ (cf. \cite[Remark 3.7]{ven02}).

\section{Selmer groups and the $\fM_H(G)$-Conjecture}\label{section3}

Let $p$ be any prime number. Throughout, let $F$ denote a global function field of characteristic coprime to $p$, and write $k$ for its constant field. Suppose that $A$ is an abelian variety defined over $F$. Let $F_\infty$ be a $p$-adic Lie extension of $F$ satisfying the following admissibility conditions:
\begin{enumerate}[(i)]
\item $F_\infty$ contains the cyclotomic $\BZ_p$-extension $F^\mathrm{cyc}$ of $F$.
\item $G = \Gal(F_\infty/F)$ is a compact $p$-adic group without $p$-torsion.
\item $F_\infty/F$ is unramified outside of a finite set of primes.
\end{enumerate}
In particular, the extension $F^{\mathrm{cyc}}$ itself satisfies these conditions. Let $S$ be a nonempty finite set of primes of $F$ containing all primes where the abelian variety $A$ has
bad reduction, as well as all primes of $F$ that ramify in $F_\infty$.

Let $\bar{F}$ denote a separable closure of $F$. For any extension $\fK$ inside $\bar{F}$, we define the $p$-primary Selmer group $\Sel(A/\fK)$ using Galois cohomology as follows:
\begin{equation}\label{1-f-1}
  \Sel(A/\fK)={\rm ker}\left(H^1(\fK,A_{p^\infty})\rightarrow\prod_w H^1(\fK_w,A)\right),
\end{equation}
where $A_{p^\infty}$ denotes the group of all $p$-power division points in $A(\bar{F})$. Note that when $\fK$ is an infinite extension of $F$, each $\fK_w$ is defined as the union of completions at $w$ of all finite extensions of $F$ contained in $\fK$. If $\fK$ is Galois over $F$, then the natural action of $\Gal(\fK /F)$ on $H^1(\fK,A_{p^\infty})$ endows $\Sel(A/\fK)$ with the structure of a $\Gal(\fK/F)$-module.\\

Let $F_S$ denote the maximal extension of $F$ inside $\bar{F}$ which is unramified outside $S$. By the choice of $S$, we know $F_\infty$ is contained in $F_S$.
For any extension $K$ of $F$ contained in $F_S$, we write $G_S(K)=\Gal(F_S/K)$. Now, for a finite extension $L/F$ and for each prime $v$ of $F$, define
\[J_v(A/L)=\bigoplus_{w\mid v}H^1(L_w,A)(p),\]
where the direct sum is taken over all primes $w$ of $L$ lying above $v$. Since $L_w$ has characteristic different from $p$, Kummer theory gives the natural identification \mbox{$H^1(L_w,A)(p) \simeq H^1(L_w, A_{p^\infty})$.} We shall make use of this identification without any further comment.

Let $\bar{L}_w$ be a separable closure of $L_w$. By choosing an embedding $\bar{L}\subset \bar{L}_w$, we may view $\Gal(\bar{L}_w/L_w)$ as a subgroup of $\Gal(\bar{L}/L)$. This gives rise to the following localisation map:
\begin{equation}\label{2-1-f-1}
  \lambda_S(A/L): H^1(G_S(L),A_{p^\infty}) \to \bigoplus_{v\in S}J_v(A/L),
\end{equation}
which is induced by restriction.

For an infinite separable extension $K$ of $F$ contained in $F_S$, we define
\[J_v(A/K)=\lim_{\longrightarrow}J_v(A/L),\]
where the direct limit is taken over all finite extensions $L$ of $F$ contained in $K$, with respect to the natural restriction maps. Correspondingly, we define the localisation map
\[\lambda_S(A/K)=\lim_{\longrightarrow}\lambda_S(A/L).\]

We will use the following result on Selmer groups, which can be proved in the same way as for elliptic curves over number fields (see \cite[Lemma 2.3]{coates}).

\begin{prop}\label{2-1-1}
Let $K$ be an extension of $F$ contained in $F_S$. Then $\Sel(A/K)$ fits into the exact sequence:
\[
0 \longrightarrow \mathrm{Sel}(A/K) \longrightarrow H^1(G_S(K), A_{p^\infty})
\stackrel{\lambda_S(A/K)}{\longrightarrow}
\bigoplus_{v \in S} J_v(A/K).
\]
\end{prop}

The remainder of this section  is devoted to  proving the $\mathfrak{M}_H(G)$-conjecture for the Pontryagin dual $X(A/F_\infty)$ of the $p$-primary Selmer group $\Sel(A/F_\infty)$. To this end, we begin with a crucial observation.

\begin{lem}\label{cdpcyc} We have
$$\mathrm{cd}_p(G_S(F^{\mathrm{cyc}}))=1.$$

\end{lem}

\begin{proof} Since the set $S$ is nonempty and $\mathrm{char}(F)\neq p$, \cite[Theorem 10.1.2 (i)(c)]{nsw} implies that
$$\mathrm{cd}_p(G_S(\bar{k}F))=1.$$
Note that $G_S(\bar{k}F)$ is a closed subgroup of $G_S(F^{\mathrm{cyc}})$, and that its quotient $\Gal(\bar{k}F/F^{\mathrm{cyc}})$ is a profinite group in which every open subgroup has an index prime to $p$. Hence, by applying \cite[Proposition 3.3.5 (i)]{nsw}, we conclude that $\mathrm{cd}_p(G_S(F^{\mathrm{cyc}}))=\mathrm{cd}_p(G_S(\bar{k}F))$, and the lemma follows.
\end{proof}

\begin{lem}\label{corankglo} The cohomology group
$H^1(G_S(F_\infty), A_{p^\infty})$ is cofinitely generated over $\Lambda(H)$. In particular, the Pontryagin dual $H^1(G_S(F_\infty), A_{p^\infty})^\vee$ is a torsion $\Lambda(G)$-module.
\end{lem}

\begin{proof}
Without loss of generality, we may assume that both $G$ and consequently $H$, are pro-$p$ and have no $p$-torsion. By Nakayama's lemma, it suffices to show that $H^1(G_S(F_\infty), A_{p^\infty})^H$ is cofinitely generated over $\BZ_p$.

Since $H^2(G_S(F^{\mathrm{cyc}}),A_{p^\infty})=H^2(G_S(F_\infty), A_{p^\infty})=0$ by Lemma \ref{cdpcyc}, the Hochschild--Serre spectral sequence
$$H^i\left(H,H^j(G_S(F_\infty),A_{p^\infty})\right)\Rightarrow  H^{i+j}(G_S(F^{\mathrm{cyc}}), A_{p^\infty})$$
degenerates to yield the following exact sequence
$$0\to H^1(H,A_{p^\infty}(F_\infty))\to H^1(G_S(F^{\mathrm{cyc}}),A_{p^\infty})\to H^1(G_S(F_\infty),A_{p^\infty})^H\to H^2(H,A_{p^\infty}(F_\infty)) \to 0.
$$
Since both $H^1(H,A_{p^\infty}(F_\infty))$ and $H^2(H,A_{p^\infty}(F_\infty))$ are cofinitely generated over $\BZ_p$ (see the proof of \cite[Theorem 1.1]{howson}), it remains to show that $H^1(G_S(F^{\mathrm{cyc}}),A_{p^\infty})$ is cofinitely generated over $\BZ_p$. By Nakayama's lemma, this is equivalent to showing that $H^1(G_S(F^{\mathrm{cyc}}),A_{p^\infty})[p]$ is finite. From the long exact sequence of Galois cohomology associated with the short exact sequence \mbox{$0\to A_p\to A_{p^\infty}\xrightarrow{p} A_{p^\infty}\to 0$}, we obtain
a surjection
\[ H^1(G_S(F_\cyc), A_{p}) \twoheadrightarrow H^1(G_S(F_\cyc), A_{p^\infty})[p].\]
Hence, it suffices to show that $H^1(G_S(F_\cyc), A_{p})$ is finite. This is equivalent to proving that
$$\corank_{\BF_p[[\Gamma]]}\left(H^1(G_S(F_\cyc), A_{p})\right)=0.$$
Since $\corank_{\BF_p[[\Gamma]]}\left(H^0(G_S(F^\cyc),A_p)\right)=0$ and $H^2(G_S(F^{\cyc}),A_p)=0$, we have
\begin{equation}\label{altsum}\corank_{\BF_p[[\Gamma]]}\left(H^1(G_S(F^\cyc),A_p)\right)
=\sum_{i=0}^2(-1)^{i+1}\corank_{\BF_p[[\Gamma]]}\left(H^i(G_S(F^\cyc),A_p)\right).
\end{equation}
Using the spectral sequence
\[ H^i\left(\Gamma, H^j(G_S(F^\cyc), A_{p})\right) \Rightarrow H^{i+j}(G_S(F),A_p), \]
we conclude that the right-hand side of (\ref{altsum}) equals
$$\sum_{i=0}^2(-1)^{i+1}\corank_{\BF_p}(H^i(G_S(F), A_p)),$$
which vanishes by a global Euler--Poincar\'{e} characteristic calculation (cf. \cite[(8.7.4)]{nsw}). This concludes the proof of the lemma.
\end{proof}

We can now prove the main result of this section.

\begin{thm}\label{MHG} Let $A$ be an abelian variety defined over a global function field $F$ of characteristic prime to $p$. Let $F_\infty/F$ be an admissible $p$-adic Lie extension. Then the Selmer group $\Sel(A/F_\infty)$ is a cofinitely generated $\Lambda(H)$-module. In particular, the $\mathfrak{M}_H(G)$-conjecture holds for $X(A/F_\infty)$.
Furthermore, $\Sel(A/F^\cyc)$ is a cofinitely generated $\BZ_p$-module.
\end{thm}

\begin{proof} Since $\Sel(A/F_\infty)$ is contained in $H^1(G_S(F_\infty), A_{p^\infty})$, the result follows immediately from Lemma \ref{corankglo}.
\end{proof}

We remark that the above theorem relies crucially on the assumption that ${\rm char}(F)\neq p$; otherwise, the key ingredient---the global Euler--Poincar\'{e} characteristic---fails, due to the fact that $p = 0$ in any ring of characteristic $p$. This also explains the $\mu = 0$ condition imposed in \cite{ochitri}.
We conclude with the following corollary concerning the generalised $\mu$-invariant, as defined in \cite{howson, ven02, BV} for finitely generated $\Lambda(G)$-modules.

\begin{cor}\label{genmu} The generalised $\mu$-invariant of  $X(A/F_\infty)$ is zero.
\end{cor}

\begin{proof}We have shown that $\Sel(A/F_\infty)$ is a cofinitely generated $\Lambda(H)$-module. By \cite[Proposition 2.3]{c-non}, this is equivalent to $X(A/F_\infty)$ being $\Sigma$-torsion, where $\Sigma$ is the Ore set introduced in Section \ref{section7}. The result then follows from, for example, the proof of \cite[Proposition 3.3 (i)]{BV} (also see \cite[Lemma 2.7]{howson}).
\end{proof}

\section{Surjectivity of localisation map}\label{section4}

In this section, we establish the surjectivity of the localisation map $\lambda_S(A/F_\infty)$ defining the Selmer group $\Sel(A/F_\infty)$. In particular, we deduce the surjectivity of the localisation map $\lambda_S(A/F^\mathrm{cyc})$.
Before giving the proof, we introduce the \textit{compact Selmer group} $\fS(A/F_\infty)$. Let $T_pA=\varprojlim_n A_{p^n}$ denote the $p$-adic Tate module of $A$. For every finite extension $L$ of $F$ contained in $F_\infty$, the compact Selmer group $\fS(A/L)$ is defined as
\[ \mathrm{ker}\Big( H^1(G_S(L), T_p A) \to \bigoplus_{w\in S} H^1(L_w, T_p A)\Big),\]
where the sum runs over all primes $w$ of $L$ lying above the set $S$. We then define the compact Selmer group over $F_\infty$ as the inverse limit
\[\fS(A/F_\infty)=\varprojlim_{L}\fS(A/L),\]
where the limit is taken over all finite extensions $L$ of $F$ contained in $F_\infty$.
We also define the continuous Iwasawa cohomology of
$T_pA$ by
$$H^i_{\mathrm{Iw}}(F_\infty,T_pA)=\varprojlim_{L} H^{i}(G_S(L),T_p A),$$
where the limit is taken over all finite Galois extensions with transition maps given by the corestriction maps on cohomology.
In particular, $\fS(A/F_\infty)$ is a submodule of $H^1_{\mathrm{Iw}}(F_\infty,T_pA)$.

We now record the following result, which will be a key ingredient in our eventual proof of surjectivity.

\begin{lem}\label{torfree} The module $H^1_{\mathrm{Iw}}(F_\infty, T_p A)$ is a torsion $\Lambda(G)$-module. Moreover, if $\dim (G)\geq 2$ or if $A_{p^\infty}(F_\infty)$ is finite, then $H^1_{\mathrm{Iw}}(F_\infty, T_p A)=0$.
\end{lem}

\begin{proof} We consider the following spectral sequence of Jannsen (see \cite[Theorem 1, Theorem 11]{jannsen}; see also \cite[Theorem 4.5.1]{LS} and \cite[Theorem 8.5.6]{nekovar})
$$E_2^{i,j}=\Ext^i_{\Lambda(G)} \left(H^j\left(G_S(F_\infty),  A_{p^\infty})\right)^\vee, \Lambda(G)\right) \Rightarrow H^{i+j}_{\mathrm{Iw}}(F_\infty,T_p A).$$
The low-degree terms of this spectral sequence yield the following exact sequence
\begin{equation}\label{jannseq1}0\to \Ext^1_{\Lambda(G)} \left(A_{p^\infty}(F_\infty)^\vee, \Lambda(G)\right)\to  H^1_{\mathrm{Iw}}(F_\infty,T_p A)\to {\rm Hom}_{\Lambda(G)} \left(H^1\left(G_S(F_\infty),  A_{p^\infty})\right)^\vee, \Lambda(G)\right).
\end{equation}
By Lemma \ref{corankglo}, we have $\rank_{\Lambda(G)}\left(H^1\left(G_S(F_\infty),  A_{p^\infty})\right)^\vee\right)=0$, so the right-most term vanishes.
Since $ \Ext^1_{\Lambda(G)} \left(A_{p^\infty}(F_\infty)^\vee, \Lambda(G)\right)$ is a torsion $\Lambda(G)$-module (cf.\, \cite[Proposition 3.5 (iii) (a)]{ven02}), it follows that $H^1_{\mathrm{Iw}}(F_\infty, T_p A)$ is a torsion $\Lambda(G)$-module. Now by an observation of Jannsen (cf. \cite[Corollary 2.6(b)(c)]{jannsen89}), if $\dim (G)\geq 2$ or if $A_{p^\infty}(F_\infty)$ is finite, then $\Ext^1_{\Lambda(G)} \left(A_{p^\infty}(F_\infty)^\vee, \Lambda(G)\right) =0$. Therefore, both the left-most and right-most terms in \eqref{jannseq1} vanish, and we conclude that $H^1_{\mathrm{Iw}}(F_\infty,T_p A)=0$, as required.
\end{proof}

\begin{thm}\label{thmsurj} Let $A$ be an abelian variety defined over a global function field $F$ of characteristic prime to $p$, and let $F_\infty/F$ be an admissible $p$-adic Lie extension. Then the localisation map $\lambda_S(A/F_\infty)$ is surjective, yielding the following short exact sequence:
\begin{equation}\label{thmsurj-f001}
  0 \longrightarrow \operatorname{Sel}(A/F_\infty) \longrightarrow H^1(G_S(F_\infty), A_{p^\infty})
\overset{\lambda_S(A/F_\infty)}{\xrightarrow{\hspace{3cm}}} \bigoplus_{v \in S} J_v(A/F_\infty) \longrightarrow 0.
\end{equation}

\end{thm}

\begin{proof} Since the conclusion remains unchanged upon replacing the base field $F$ by a larger field contained in $F_\infty$, we may assume that $G$ is pro-$p$. The Poitou--Tate sequence (see \cite[Proof of Theorem 8.7.9]{nsw}) gives the exact sequence
\begin{align*}0\to \Sel(A/F_\infty)\to H^1(G_S(F_\infty),A_{p^\infty})\stackrel{\lambda_S(A/F_\infty)}{\longrightarrow}\bigoplus_{v\in S}J_v(A/F_\infty)\to\fS(A^*/F_\infty)^\vee \to 0,
\end{align*}
where the final ``$0$'' term follows from Lemma \ref{cdpcyc}. Thus, it suffices to show that $\fS(A^*/F_\infty)=0$. By definition, $\fS(A^*/F_\infty)$ is contained in $H^1_{\mathrm{Iw}}(F_\infty,T_pA^*)$. If ${\rm dim}(G)\geq2$, then Lemma \ref{torfree} implies that $H^1_{\mathrm{Iw}}(F_\infty,T_pA^*)=0$, and hence $\fS(A^*/F_\infty)=0$.

It remains to consider the case $F_\infty=F^{\mathrm{cyc}}$. This case is slightly more delicate, as we do not know a priori whether $A_{p^\infty}(F^{\mathrm{cyc}})$ is finite (see Remark \ref{remUlmer} for the elliptic curve case).
 By replacing $F$ with a finite extension if necessary, we may assume that every prime in $S$ remains inert in $F^{\mathrm{cyc}}/F$. For each $v\in S$, we denote by $v$ also the unique prime of $F^{\mathrm{cyc}}$ lying above $v$, and identify the Galois groups $\Gamma=\Gal(F^{\mathrm{cyc}}/F)$ with $\Gal(F^{\mathrm{cyc}}_v/F_v)$. On the global side, the Jannsen spectral sequence (see \cite[Theorem 1, Theorem 11]{jannsen}; see also \cite[Theorem 4.5.1]{LS} and \cite[Theorem 8.5.6]{nekovar})) gives
$$\Ext^i_{\Lambda(\Gamma)} \left(H^j\left(G_S(F^\cyc),  A^*_{p^\infty})\right)^\vee, \Lambda(\Gamma)\right) \Rightarrow H^{i+j}_{\mathrm{Iw}}(F^\cyc,T_p A^*).$$
The low-degree terms yield the following exact sequence
\begin{equation}\label{jannseq}0\to \Ext^1_{\Lambda(\Gamma)} \left(A^*_{p^\infty}(F^\cyc)^\vee, \Lambda(G)\right)\to  H^1_{\mathrm{Iw}}(F_\infty,T_p A^*)\to \Ext^0_{\Lambda(\Gamma)} \left(H^1\left(G_S(F^\cyc),  A^*_{p^\infty})\right)^\vee, \Lambda(\Gamma)\right).
\end{equation}
By Lemma \ref{corankglo}, we have $\rank_{\Lambda(\Gamma)}\left(H^1\left(G_S(F_\infty),  A_{p^\infty})\right)^\vee\right)=0$, so the right-most term vanishes.

For each $v\in S$, there is also a local version of Jannsen's spectral sequence (see \cite[Theorem 11]{jannsen}, \cite[Theorem 4.2.2]{LS} or \cite[Theorem 5.2.6]{nekovar}) for $A^*$:
$$\Ext^i_{\Lambda(\Gamma)} \left(H^j\left(F^{\mathrm{cyc}}_v,  A_{p^\infty}^*)\right)^\vee, \Lambda(\Gamma)\right) \Rightarrow H^{i+j}_{\mathrm{Iw}}(F^{\mathrm{cyc}}_v,T_p A^*).$$
Let us simplify notation by writing $\mathrm{E}^i(-) = \Ext^i_{\Lambda(\Gamma)} \left(-, \Lambda(\Gamma)\right)$.
The low-degree terms give an exact sequence:
$$0\to \mathrm{E}^1\left(A^*_{p^\infty}(F^{\mathrm{cyc}}_v)^\vee\right)\to  H^1_{\mathrm{Iw}}(F^{\mathrm{cyc}}_v,T_p A^*)\to \mathrm{E}^0 \left(H^1(F^{\mathrm{cyc}}_v,  A^*_{p^\infty})^\vee\right).$$
The global and local spectral sequences are functorial   \cite[Theorem 4.5.1]{LS}, which yields the following commutative diagram with exact rows:
\begin{small}
$$
\xymatrix{
 0  \ar[r] & \mathrm{E}^1 \left(A^*_{p^\infty}(F^{\mathrm{cyc}})^\vee\right) \ar[d]_{f} \ar[r]^{} & H^1_{\mathrm{Iw}}(F^{\mathrm{cyc}},T_p A^*)\ar[d]_{g}  \ar[r]  &  0 &  & \\
0  \ar[r]& \bigoplus_{v\in S}\mathrm{E}^1 \left(A^*_{p^\infty}(F^{\mathrm{cyc}}_v)^\vee\right)  \ar[r]^{} & \bigoplus_{v\in S}H^1_{\mathrm{Iw}}(F^{\mathrm{cyc}}_v,T_p A^*)\ar[r] & \bigoplus_{v\in S}\mathrm{E}^0 \left(H^1(F^{\mathrm{cyc}}_v,  A^*_{p^\infty})^\vee\right) \ar[r] & 0,}
$$
\end{small}
where the vertical maps are induced by the natural localisation maps. Note that $\ker g= \fS(A^*/F^{\mathrm{cyc}})$. So it remains to show that $f$ is injective. Since $S$ is non-empty, and $f=\bigoplus_{v\in S} f_v$, it suffices to show that each
$$f_v: \mathrm{E}^1 \left(A^*_{p^\infty}(F^{\mathrm{cyc}})^\vee\right)  \to  \mathrm{E}^1 \left(A^*_{p^\infty}(F^{\mathrm{cyc}}_v)^\vee\right) $$
is injective. But since both $A^*_{p^\infty}(F^{\mathrm{cyc}})^\vee$ and  $A^*_{p^\infty}(F^{\mathrm{cyc}}_v)^\vee$ are finitely generated $\BZ_p$-modules, \cite[Corollary 5.5.7]{nsw} implies that $f_v$ is given by
$$\Hom_{\BZ_p}(A^*_{p^\infty}(F^{\mathrm{cyc}})^\vee, \BZ_p)\to \Hom_{\BZ_p}(A^*_{p^\infty}(F^{\mathrm{cyc}}_v)^\vee, \BZ_p),$$
 which is injective, as it is induced by the inclusion $A^*_{p^\infty}(F^{\mathrm{cyc}})\hookrightarrow A^*_{p^\infty}(F^{\mathrm{cyc}}_v)$. This completes the proof.
\end{proof}

\begin{remark}\label{remUlmer} In the case where $A=E$ is an elliptic curve without complex multiplication, the proof of Theorem \ref{thmsurj} can be simplified. By a theorem of Deuring, this is equivalent to the condition that $E$ is not defined over a finite field; equivalently, $E$ is non-isotrivial. In particular, we do not need to treat the case $F_\infty=F^{\mathrm{cyc}}$ separately, since Lemma \ref{torfree} holds without requiring the assumption $\dim(G)\geq 2$. One way to see this is via a result of Ulmer, which asserts that $E(\bar{k}F)$ is a finitely generated abelian group (see \cite[\S6.3]{ulmer}). Since $F^{\rm cyc}$ is a subfield of $\bar{k}F$, it follows that $E(F^{\rm cyc})$ is a finitely generated abelian group. Consequently,
$E_{p^\infty}(F^{\rm cyc})$ is finite. It implies that $E_{p^\infty}(F^{\rm cyc})$ is pseudo-null as a $\Lambda(G)$-module, and therefore $H^1_{\mathrm{Iw}}(F_\infty, T_p E)$ injects into the $\Ext^0$-term of the corresponding spectral sequence.
\end{remark}

\section{Akashi series}\label{section5}

Let $\mathcal{Q}(\Gamma)$ denote the field of fractions of $\Lambda(\Gamma)$. To each $M\in \fM_H(G)$, we associate a non-zero element $\mathrm{Ak}(M)\in\mathcal{Q}(\Gamma)$, which is uniquely determined up to multiplication by a unit in $\Lambda(\Gamma)$. It can be shown (see \cite[Lemma 3.1]{c-non}) that the homology groups
$$H_i(H,M)\qquad \text{for }i\geq0$$
are finitely generated torsion $\Lambda(\Gamma)$-modules. We fix, from now on, a topological generator of $\Gamma$, and identify $\Lambda(\Gamma)$ with the formal power series ring $\BZ_p[[T]]$ by mapping this topological generator to $1 + T$.  Let $f_{M,i}\in \Lambda(\Gamma)$ denote a characteristic power series of $H_i(H,M)$, and define
\begin{equation}\label{3-f-3}
  \mathrm{Ak}(M)=\prod_{i\geq0}f_{M,i}^{(-1)^i}.
\end{equation}
This product is finite because $H_i(H,M)=0$ for $i\geq \dim (H)$, and it is well defined up to multiplication by a unit in $\Lambda(\Gamma)$, since each $f_{M,i}$ is. We call $\mathrm{Ak}(M)$ the \textit{Akashi series} of $M$. If $Y$ is a discrete $p$-primary $G$-module such that $M=Y^\vee\in \fM_H(G)$, then we also refer to $\mathrm{Ak}(M)$ as the Akashi series of $Y$, and by abuse of notation, write it as $\mathrm{Ak}(Y)$. Via the duality
$$H_i(H,M) = H^i(H,Y)^\vee,$$
we will henceforth work with $H$-cohomology groups in computating the Akashi series, without further mention. This definition generalises the notion of the characteristic power series. Specifically,
if $M$ is a finitely generated torsion $\Lambda(\Gamma)$-module (i.e. when $H=\{1\}$), then the Akashi series of $M$ coincides with the $\Lambda(\Gamma)$-characteristic power series of $M$, and we denote $\mathrm{Ak}(M)$ simply by $f_M$, or $f_Y$.

\medskip

\begin{lem}\label{lemakashi} Given a short exact sequence
  \begin{equation}\label{akmul}
    0\to M_1\to M_2 \to M_3\to 0
  \end{equation}
in the category $\fM_H(G)$, the Akashi series satisfy $\mathrm{Ak}(M_2) = \mathrm{Ak}(M_1)\cdot \mathrm{Ak}(M_3)$.
\end{lem}

\begin{proof}
This follows by considering the long exact sequence in $H$-homology derived from \eqref{akmul}, and applying the multiplicativity of $\Lambda(\Gamma)$-characteristic series in a short exact sequence.
\end{proof}

\begin{remark} We note that the original definition of $\fM_H(G)$ in \cite{coates-sujatha-schneider} differs from the one used in \cite{c-non} and in this paper. One consequence of this discrepancy is that, in general,  $\mathrm{Ak}(M) \neq \mathrm{Ak}(M/M(p))$, as shown in \cite[Lemma 4.1]{coates-sujatha-schneider}. This is because $M(p)$ need not be finitely generated over $\Lambda(H)$, and hence its homology group may not be finitely generated over $\BZ_p$. Therefore, even though its homology group is annihilated by some power of $p$, we cannot conclude that they are finite. In general, $\mathrm{Ak}(M(p))$ may be non-trivial. In the case where $G$ is a pro-$p$ group without $p$-torsion, we have $\mathrm{Ak}(M(p))=p^{\mu(M)}$ (see \cite[1.6]{AW}), where $\mu(M)$ denotes the generalised $\mu$-invariant of $M$, as defined in \cite{howson, ven02, BV}.
\end{remark}

By virtue of Theorem \ref{thmsurj}, we have a short exact sequence
$$  0\to \Sel(A/F_\infty)\to H^1(G_S(F_\infty),A_{p^\infty})\to \bigoplus_{v\in S} J_v(A/F_\infty)\to 0.$$
Since $\Sel(A/F_\infty)$ and $H^1(G_S(F_\infty),A_{p^\infty})$ are cofinitely generated $\Lambda(H)$-modules by Lemma \ref{corankglo} and Theorem \ref{MHG}, so is $\bigoplus_{v\in S} J_v(A/F_\infty)$. Hence, we may apply Lemma \ref{lemakashi} to deduce
\begin{equation}\label{akses}
\mathrm{Ak}\Big( H^1(G_S(F_\infty),A_{p^\infty})\Big) = \mathrm{Ak}\Big(\Sel(A/F_\infty)\Big) \mathrm{Ak}\left(\bigoplus_{v\in S} J_v(A/F_\infty)\right).
\end{equation}
Thus, determining $\mathrm{Ak}\Big(\Sel(A/F_\infty)\Big)$  reduces to computing
\[\mathrm{Ak}\Big( H^1(G_S(F_\infty),A_{p^\infty})\Big)\quad \textrm{and}\quad \mathrm{Ak}\left(\bigoplus_{v\in S} J_v(A/F_\infty)\right),\]
which are the subjects of the next two lemmas.

\begin{lem} \label{Akglobal}
 $\mathrm{Ak}\Big( H^1(G_S(F_\infty),A_{p^\infty})\Big)= f_{H^1(G_S(F^\cyc), A_{p^\infty})} \cdot \prod_{i\geq 1}\left(f_{H^i\left(H, A_{p^\infty}(F_\infty)\right)}\right)^{(-1)^i}$.
\end{lem}

\begin{proof}
Since $H^2(G_S(F^{\mathrm{cyc}}),A_{p^\infty})=H^2(G_S(F_\infty), A_{p^\infty})=0$ by Lemma \ref{cdpcyc}, the spectral sequence
$$H^i\left(H,H^j(G_S(F_\infty),A_{p^\infty})\right)\Rightarrow  H^{i+j}(G_S(F_\infty), A_{p^\infty})$$
degenerates and yields the following exact sequence
\begin{equation}\label{beta}0\to H^1(H,A_{p^\infty}(F_\infty))\to H^1(G_S(F^{\mathrm{cyc}}),A_{p^\infty})\to H^1(G_S(F_\infty),A_{p^\infty})^H\to H^2(H,A_{p^\infty}(F_\infty)) \to 0
\end{equation}
together with isomorphisms
\begin{equation}\label{isom+2}H^i\left(H,H^1(G_S(F_\infty),A_{p^\infty})\right) \simeq H^{i+2}(H,A_{p^\infty}(F_\infty))
\end{equation}
for $i\geq 1$. The lemma now follows from a direct calculation using (\ref{beta}) and (\ref{isom+2}).
\end{proof}

Since $\mathrm{char}(F_v)\neq p$ for each $v\in S$, it follows from a theorem of Iwasawa \cite[Theorem 7.5.3]{nsw} that the decomposition subgroup $G_w$ of any prime $w$ of $F_\infty$ lying above $v$ has dimension at most $2$. Let $S'$ denote the subset of $S$ consisting of primes whose inertia group in $G$ is infinite. This is equivalent to requiring that the decomposition subgroup $G_w$ has dimension exactly $2$. Indeed, since the residue characteristic is prime to $p$ and $G$ is a $p$-adic Lie group, the wild inertia subgroup is finite. Moreover,  the maximal unramified $p$-extension of $F_v$ is $F_v^{\mathrm{cyc}}$, and the maximal tamely ramified extension of $F_v$ is a topologically cyclic extension over the maximal unramified extension.

\begin{lem} \label{AKlocal}
$\mathrm{Ak}\left(\bigoplus_{v\in S} J_v(A/F_\infty)\right) = \prod_{v\in S\setminus S'} f_{J_v(A/F^{\mathrm{cyc}})}$
\end{lem}

\begin{proof}
For each $v\in S$, we have
\begin{align*}J_v(A/F_\infty)&=\bigoplus_{w\mid v}\mathrm{Ind}_{H_w}^H H^1(F_{\infty, w}, A_{p^\infty}),
\end{align*}
where the sum is over primes $w$ of $F^{\mathrm{cyc}}$ lying above $v$, and we also write $w$ for a fixed prime in $F_\infty$ lying above $w$.
We claim that $J_v(A/F_\infty)=0$ for $v\in S'$. Indeed, let $w$ be a prime of $F_\infty$ above $v$, so that $\dim G_w= 2$ by  assumption. Then Iwasawa's theorem (\cite[Theorem 7.5.3]{nsw}) implies that $F_{\infty, w}$ has no nontrivial $p$-extensions. Hence $ H^1(F_{\infty, w}, A_{p^\infty})=0$ and the claim follows.

Now consider $v\in S\setminus S'$. To compute $\mathrm{Ak}\left(J_v(A/F_\infty)\right)$,
let $w$ be a prime of $F^{\mathrm{cyc}}$ above $v$, and fix a prime of $F_\infty$ above $w$, which we again denote by $w$. Let $H_w$ denote the decomposition subgroup of $H$ at $w$. By Shapiro's lemma, we have
\begin{equation}\label{shapiro}H^i(H, J_v(A/F_\infty))\simeq \bigoplus_{w\mid v}H^i(H_w, H^1(F_{\infty, w},A_{p^\infty})).
\end{equation}
Since $v\nmid p$, the inertia group at $v$ is finite, so $H_w$ is finite. As $G$ has no $p$-torsion, the order of $H_w$ is prime to $p$, and thus $\mathrm{cd}_p(H_w)=0$. Therefore, the right-hand side of \eqref{shapiro} vanishes for all
$i\geq 1$, implying that
\[H^i(H,J_v(A/F_\infty))=0\quad\textrm{for all $i\geq1$.}\]
For the term $J_v(A/F_\infty)^H$, we observe that the restriction map
\[  J_v(A/F^\cyc) \rightarrow J_v(A/F_\infty)^H\]
is an isomorphism, due to the fact that $\mathrm{cd}_p(H_w)=0$. In conclusion,
\[ \mathrm{Ak}\Big(J_v(A/F_\infty)\Big) = \left\{
  \begin{array}{ll}
    f_{J_v(A/F^\cyc)} & \hbox{if $v\in S\setminus S'$,} \\
    0, & \hbox{if $v\in S'$,}
  \end{array}
\right. \]
which implies the lemma.
\end{proof}

\bigskip

 \begin{thm}\label{akashithm}  Let $A$ be an abelian variety defined over a global function field $F$ with $\mathrm{char}(F)\neq p$, and let $F_\infty/F$ be an admissible $p$-adic Lie extension. Then
$$\mathrm{Ak}\left(\Sel(A/F_\infty)\right)=f_{\Sel(A/F^{\mathrm{cyc}})}\prod_{i\geq 1}\left(f_{H^i\left(H, A_{p^\infty}(F_\infty)\right)}\right)^{(-1)^i}\prod_{v\in  S'} f_{J_v(A/F^{\mathrm{cyc}})}.$$
\end{thm}

\begin{proof}
Combining the computations from Lemmas \ref{Akglobal} and \ref{AKlocal} with Equation (\ref{akses}), we obtain
\[f_{H^1(G_S(F^\cyc), A_{p^\infty})}  \prod_{i\geq 1}\left(f_{H^i\left(H, A_{p^\infty}(F_\infty)\right)}\right)^{(-1)^i} = \mathrm{Ak}\left(\Sel(A/F_\infty)\right)\prod_{v\in S\setminus S'} f_{J_v(A/F^{\mathrm{cyc}})}. \]
By Theorem \ref{thmsurj}, we have the short exact sequence
$$0\to \Sel(A/F^\cyc)\to H^1(G_S(F^\cyc),A_{p^\infty})\to \bigoplus_{v\in S} J_v(A/F^\cyc)\to 0,$$
where all modules are cofinitely generated over $\BZ_p$, by Lemma \ref{corankglo} and Theorem \ref{MHG}. Therefore,
\[ f_{H^1(G_S(F^\cyc), A_{p^\infty})} =  f_{\Sel(A/F^{\mathrm{cyc}})}\prod_{v\in S} f_{J_v(A/F^{\mathrm{cyc}})}.\]
Substituting this into the earlier expression, we obtain the required conclusion of the theorem.
\end{proof}

We record an interesting corollary of our Theorem \ref{akashithm} which serves as an analogue of a result of Howson \cite{howson} (also see \cite{CH}).

\begin{cor} \label{Howson rank}
Retain the assumptions of Theorem \ref{akashithm}. Suppose further that $G$ is pro-$p$ with dimension at least $2$. Then the following formula holds:
\begin{equation*}
\begin{split} \mathrm{rank}_{\Lambda(H)} X(A/F_\infty) = \mathrm{rank}_{\BZ_p}X(A/F^\cyc) - \mathrm{rank}_{\BZ_p}A_{p^\infty}(F^\cyc)^\vee\\ +\sum_{w\in S'(F^\cyc)} \mathrm{rank}_{\BZ_p} A_{p^\infty}(F^{\cyc}_w)^\vee.
\end{split}
\end{equation*}
\end{cor}

\begin{proof}
In the course of the proof, whenever $N$ is a finitely generated torsion $\Lambda(\Gamma)$-module with vanishing $\mu$-invariant, we shall represent the characteristic series $f_N$ by a distinguished polynomial. In this context, it makes sense to speak of $\deg f_N$. In particular, $\rank_{\BZ_p}(N)=\deg f_N$. Now, suppose $M$ is a $\Lambda(G)$-module that is finitely generated over $\Lambda(H)$. Then $H_i(H,M)$ is finitely generated over $\BZ_p$ for every $i$, and we have
\begin{equation}\label{akashi-f1}
  \rank_{\Lambda(H)}(M) = \sum_{i\geq 0}(-1)^i\rank_{\BZ_p} H_i(H,M)
\end{equation}
(cf. \cite[Theorem 1.1]{howson}). Combining this with the previous observation, we see that $\rank_{\Lambda(H)}(M)$ can be computed from its Akashi series
 \[ \rank_{\Lambda(H)}(M) = \sum_{i\geq 0}(-1)^i\deg f_{H_i(H,M)}. \]
Now recall the conclusion of Theorem \ref{akashithm}, which we rewrite as
$$\mathrm{Ak}\left(\Sel(A/F_\infty)\right)=\mathrm{Ak}\left(A_{p^\infty}(F_\infty)\right)f_{\Sel(A/F^{\mathrm{cyc}})}f_{A_{p^\infty}(F^\cyc)}^{-1}\prod_{v\in  S'} f_{J_v(A/F^{\mathrm{cyc}})}.$$
Since $\dim H\geq 1$ by assumption, we know that $\rank_{\Lambda(H)}(A_{p^\infty}(F_\infty)^\vee)=0$. It remains to verify the identity
\[ \deg f_{J_v(A/F^{\mathrm{cyc}})} =\sum_{w|v}\rank_{\BZ_p} A_{p^\infty}(F^{\cyc}_w)^\vee\]
for $v\in S'$, where the sum runs over all primes $w$ of $F^\cyc$ above $v$. To this end, note that $J_v(A/F^{\mathrm{cyc}})= \oplus_{w|v}H^1(F^\cyc_w, A_{p^\infty})$.
For $v\in S'$, the proof of Lemma 5.4 shows that $H^1(F^\cyc_w, A_{p^\infty}) = H^1(H_w, A_{p^\infty}(F_{\infty,w}))$. Since $\dim H_w\geq 1$ and $\rank_{\BZ_p[[H_w]]}(A_{p^\infty}(F_{\infty,w})^\vee)=0$, it follows from a similar formula as \eqref{akashi-f1} and ${\rm cd}_p(H_w)=1$ that
\[ \rank_{\BZ_p}H^1(H_w, A_{p^\infty}(F_{\infty,w}))^\vee = \rank_{\BZ_p}H^0(H_w, A_{p^\infty}(F_{\infty,w}))^\vee = \rank_{\BZ_p} A_{p^\infty}(F^{\cyc}_w)^\vee.\]
This completes the proof of the corollary.
\end{proof}

\section{Control Theorem}\label{section6}

In this section, we will prove the following control theorem, which enables us to relate the order of vanishing at $T=0$ of the characteristic power series of $\Sel(A/F^\mathrm{cyc})$ to the $\BZ_p$-corank of the Selmer group over $\Sel(A/F)$ over the base field $F$.

\begin{thm}\label{controlthm} Let $A$ be an abelian variety defined over a global function field $F$ of characteristic different from $p$. Let $F^{\mathrm{cyc}}$ denote the cyclotomic $\BZ_p$-extension of $F$. Write $\Gamma_n=\Gal(F^{\rm cyc}/F_n)$, where $F_n$ is the $n$-th layer of the extension $F^{\mathrm{cyc}}/F$ such that $[F_n:F]=p^n$. Then the restriction map
$$\Sel(A/F_n)\to \Sel(A/F^{\mathrm{cyc}})^{\Gamma_n}$$
has a finite kernel and cokernel, which are bounded independently of $n$.
\end{thm}

\begin{proof} We follow closely the method used in the proof of the control theorem for fine Selmer groups over number fields, as presented in \cite{LimDM}. Consider the following commutative diagram with exact rows
$$
  \xymatrix{
  0  \ar[r]^{} & \Sel(A/F_n) \ar[d]_{\alpha_n} \ar[r]^{} & H^1(G_S(F_n),A_{p^\infty}) \ar[d]_{\beta_n} \ar[r]^{} & \bigoplus_{v\in S}J_v(A/F_n)\ar[d]^{\gamma_n} &   \\
  0 \ar[r]^{} & \Sel(A/F^{\rm cyc})^{\Gamma_n} \ar[r]^{} & H^1(G_S(F^{\rm cyc}),A_{p^\infty})^{\Gamma_n} \ar[r]_{} & \bigoplus_{v\in S}J_v(A/F^{\rm cyc})^{\Gamma_n}. &    }
$$
Since $\Gamma_n$ has $p$-cohomological dimension $1$, the inflation-restriction sequence gives:
\[\ker \beta_n=H^1(\Gamma_n, A_{p^\infty}(F^{\mathrm{cyc}}))\quad \textrm{and}\quad \coker \beta_n=0.\]
Let
$M=\left(A_{p^\infty}(F^{\mathrm{cyc}})\right)^\vee$. Then $M$ is a finitely generated $\BZ_p$-module with a continuous $\Gamma$-action, hence a finitely generated  torsion $\Lambda(\Gamma_n)$-module for every $n$. By \cite[Proposition 5.3.20]{nsw},
$$\rank_{\BZ_p}M_{\Gamma_n}=\rank_{\BZ_p}M^{\Gamma_n}< \infty.$$
By the Mordell--Weil theorem, $M_{\Gamma_n}=\left(A_{p^\infty}(F_n)\right)^\vee$ is finite, so  $\rank_{\BZ_p}M_{\Gamma_n}=0$. It follows that $M^{\Gamma_n}$ is also finite. Note that
$$H^1(\Gamma_n,A_{p^\infty}(F^{\mathrm{cyc}}))=\left(A_{p^\infty}(F^{\mathrm{cyc}})\right)_{\Gamma_n},$$
which is the Pontryagin dual of $M^{\Gamma_n}$. As $M$ is a finitely generated $\BZ_p$-module, the order of $M^{\Gamma_n}$ is bounded by the order of $M(p)$, which is independent of $n$. It remains to show that $\ker \gamma_n$ is finite and bounded independently of $n$. By the snake lemma, this will imply the same uniform boundedness for $\ker \alpha_n$ and $\coker \alpha_n$. To analyse $\gamma_n$, observe that $v\in S$ does not divide $p$, it suffices to show that for each prime $w$ of $F^{\mathrm{cyc}}$ lying above $v\in S$, the group $H^1(\Gamma_w, A_{p^\infty}(F^{\mathrm{cyc}}_w))$ is finite and bounded independently of $n$. This can be achieved by replacing $\Gamma_n$ with $\Gamma_w$, and $M$ with $\left(A_{p^\infty}(F^{\mathrm{cyc}}_w)\right)^\vee$, which remains a finitely generated $\BZ_p$-module, and then applying the same argument as above.
\end{proof}

In Theorem \ref{MHG}, we showed  that $X(A/F^{\mathrm{cyc}})$ is a finitely generated torsion $\Lambda(\Gamma)$-module with $\mu$-invariant zero. Thus, the structure theorem for finitely generated $\Lambda(\Gamma)$-modules implies that there exists a pseudo-isomorphism
\begin{equation}\label{Xcyc}X(A/F^{\mathrm{cyc}})\sim \bigoplus_{i=1}^k \Lambda(\Gamma)/(f_i^{a_i}),
\end{equation}
where we identify $\Lambda(\Gamma)\simeq \BZ_p[[T]]$ using our fixed topological generator of $\Gamma$, and each $f_i$ for $i=1,\ldots , k$, is a distinguished irreducible polynomial of degree $\lambda_i$. Consequently, $\Lambda(\Gamma)/(f_i^{a_i})$ is isomorphic to $\BZ_p^{a_i \lambda_i }$ as an abelian group, and
$$\rank_{\BZ_p}\left(X(A/F^{\mathrm{cyc}})\right)=\lambda=\sum_{i=1}^k a_i \lambda_i.$$
Moreover, we have $\prod_{i=1}^k f_i ^{a_i}= f_{\Sel(A/F^\mathrm{cyc})}$ up to a unit in $\Lambda(\Gamma)$.

\bigskip

We will make use of the following analogue of Greenberg's semi-simplicity conjecture (\cite[Conjecture 1.12]{greenberg}), originally posed for elliptic curves over number fields, but here considered for abelian varieties over global function fields.

\begin{conj}[Greenberg's semi-simplicity conjecture for $A/F$]\label{greenberg} Let $A$ be an abelian variety defined over a global function field $F$ with ${\rm char}(F)\neq p$. Then the action of $\Gamma$ on $X(A/F^{\mathrm{cyc}})$ is completely reducible, that is, in the decomposition \eqref{Xcyc}, we have $a_i=1$ for all $i$.
\end{conj}

We then have the following result:

\begin{prop}\label{controlprop} Let $A$ be an abelian variety defined over a global function field $F$. Then
$$\ord_{T=0}(f_{\Sel(A/F^{\mathrm{cyc}})})\geq \corank_{\BZ_p}(\Sel(A/F)).$$
Moreover, equality holds if Conjecture \ref{greenberg} is true.
\end{prop}

\begin{proof} By Theorem \ref{controlthm}, the natural restriction map $\Sel(A/F)\to \Sel(A/F^{\mathrm{cyc}})^{\Gamma}$ has a finite kernel and cokernel, which by Pontryagin duality implies the same for the map $X(A/F^{\mathrm{cyc}})_{\Gamma} \to X(A/F)$.
It follows that
$$\rank_{\BZ_p}(X(A/F^{\mathrm{cyc}})_{\Gamma})=\corank_{\BZ_p}(\Sel(A/F)).$$
Next, note that \[\rank_{\BZ_p}(X(A/F^{\mathrm{cyc}})_{\Gamma})=\rank_{\BZ_p}\left(X(A/F^{\mathrm{cyc}})/T X(A/F^{\mathrm{cyc}})\right).\]
Since, for any distinguished polynomial $f(T)$, we have $\Lambda(\Gamma)/(f, T)  \simeq \BZ_p/ f(0) \BZ_p,$ it follows that the $\BZ_p$-rank of $\Lambda(\Gamma)/(f, T)$ is $1$ if and only if $T\mid f$.
Therefore, from \eqref{Xcyc}, the $\BZ_p$-rank of $X(A/F^{\mathrm{cyc}})/T X(A/F^{\mathrm{cyc}})$ is equal to the number of polynomials $f_i$ that are equal to $T$. Assuming Conjecture \ref{greenberg}, this is precisely to the power of $T$ dividing $f_{\Sel(A/F^{\mathrm{cyc}})}$, and the result follows.
\end{proof}

\section{The order of vanishing of the characteristic element}\label{section7}

In this section, we prove our main result concerning the order of vanishing of the characteristic element associated with the Selmer group over a $p$-adic Lie extension. We begin by recalling the definition of these characteristic elements.

As before, let $G$ denote a compact $p$-adic Lie group without $p$-torsion, and let $H$ be a closed normal subgroup of $G$ such that $\Gamma:=G/H\simeq \BZ_p$. First, we introduce the characteristic element $\xi_M$ attached to $M\in \mathfrak{M}_H(G)$. To this end, we recall from \cite{c-non} an equivalent definition of the category $\mathfrak{M}_H(G)$. A multiplicatively closed subset $\Sigma\subset\Lambda(G)$ is said to be a left and right Ore set if, for each $f\in \Sigma$ and $x\in \Lambda(G)$, there exist $g_1, g_2\in \Sigma$ and $y_1, y_2\in \Lambda(G)$ such that
$$f y_1=x g_1 \quad \text{and} \quad y_2 f=g_2 x.$$
We define
$$\Sigma=\{f\in \Lambda(G): \Lambda(G)/\Lambda(G)f \text{ is a finitely generated $\Lambda(H)$-module}\}.$$
It is shown in \cite[Theorem 2.4]{c-non} that $\Sigma$ is a multiplicatively closed left and right Ore set in $\Lambda(G)$. Since $p\notin \Sigma$, we also consider the enlarged set
$$\Sigma^*= \cup_{n\geq 0}p^n \Sigma,$$
which is again a multiplicatively closed left and right Ore set in $\Lambda(G)$. We may therefore localise $\Lambda(G)$ at $\Sigma$ and at $\Sigma^*$ to obtain $\Lambda(G)_\Sigma$ and $\Lambda(G)_{\Sigma^*}=\Lambda(G)_\Sigma \left[\frac{1}{p}\right]$, respectively. It follows from \cite[Proposition 2.3]{c-non} that a finitely generated $\Lambda(G)$-module $M$ lies in $ \mathfrak{M}_H(G)$ if and only if $M$ is $\Sigma^*$-torsion.
Since the Iwasawa algebra $\Lambda(G)$ is Noetherian and $G$ has no element of order $p$, the global dimension of $\Lambda(G)$ is finite. From \cite{c-non}, $K_0(\Lambda(G))$ can be identified with the Grothendieck group of the category of all finitely generated
$\Lambda(G)$-modules.
From the localisation sequence in $K$-theory for the Ore set $\Sigma^*$, we obtain the following exact sequence
$$\cdots \to K_1(\Lambda(G))\to K_1(\Lambda(G)_{\Sigma^*})\stackrel{\partial_G}{\longrightarrow}K_0(\mathfrak{M}_H(G))\to K_0(\Lambda(G))\to K_0(\Lambda(G)_{\Sigma^*})\to 0.$$
By \cite[Proposition 3.4]{c-non}, the map $\partial_G$ is surjective.
 In light of this result, we can define a characteristic element for each
 $M\in\mathfrak{M}_H(G)$.

\begin{defn} Given $M\in\mathfrak{M}_H(G)$, a characteristic element for $M$ is any element $\xi_M\in K_1(\Lambda(G)_{\Sigma^*})$ satisfying
$$\partial_G(\xi_M)=[M],$$
where $[M]$ denotes the class of $M$ in $K_0(\Lambda(G))$.
\end{defn}

We fix an algebraic closure $\overline{\BQ}_p$ of $\BQ_p$. Let $$\rho: G\to \GL_n(\mathcal{O})$$
be an Artin representation, where $\mathcal{O}$ is the ring of integers of a finite extension of $\BQ_p$ in $\overline{\BQ}_p$. We write $\Lambda_{\mathcal{O}}(\Gamma)$ for the Iwasawa algebra of $\Gamma$ with coefficients in $\mathcal{O}$, which we identify with the formal power series ring $\mathcal{O}[[T]]$ by sending a fixed topological generator of $\Gamma$ to $1+T$. Let $\mathcal{Q}_{\mathcal{O}}(\Gamma)$ denote the field of fractions of $\Lambda_{\mathcal{O}}(\Gamma)$. If $M$ is a finitely generated $\Lambda(G)$-module, we define $M_{\mathcal{O}}=M\otimes_{\BZ_p} \mathcal{O}$, and set
\begin{equation}\label{twist}\mathrm{tw}_{\rho}(M)=M_{\mathcal{O}}\otimes_{\mathcal{O}} \mathcal{O}^n
\end{equation}
where $\CO^n$ is endowed with the action of  $G$ via $\rho$, and $G$ acts diagonally on the tensor product.

The representation $\rho$ defines a continuous group homomorphism
$$G\to (M_n(\mathcal{O})\otimes_{\BZ_p}\Lambda(\Gamma))^\times$$
which maps each $g$ to $\rho(g)\otimes \bar{g}$, where $\bar{g}$ denotes the image of $g$ in $\Gamma$. By \cite[Lemma 3.3]{c-non}, this extends to a ring homomorphism
$$\Lambda(G)_{\Sigma^*}\to M_n(\mathcal{Q}_\mathcal{O}(\Gamma)).$$ By the functoriality of $K_1$-groups, this in turn induces a group homomorphism
$$K_1(\Lambda(G)_{\Sigma^*})\to K_1(M_n(\mathcal{Q}_\mathcal{O}(\Gamma))).$$
We define $\Phi_\rho$ to be the composite map
\begin{equation}\label{Phi_rho}\Phi_\rho: K_1(\Lambda(G)_{\Sigma^*})\to K_1(M_n(\mathcal{Q}_\mathcal{O}(\Gamma))) \stackrel{\text{Morita}}{\simeq} K_1(\mathcal{Q}_\mathcal{O}(\Gamma))\simeq \mathcal{Q}_\mathcal{O}(\Gamma)^\times,
\end{equation}
where the first isomorphism is given by Morita's equivalence. We also identify $\mathcal{Q}_\mathcal{O}(\Gamma)^\times$ with $\mathcal{Q}_\mathcal{O}(T)^\times$, where $\mathcal{Q}_\mathcal{O}(T)$ denotes the field of fractions of $\mathcal{O}[[T]]$.

\begin{defn}[Burns \cite{burns}] Let $\xi\in K_1(\Lambda(G)_{\Sigma^*})$, and let $\rho$ be an Artin representation of $G$. Then there exists an integer $r_\rho(\xi)$ and $g(T)\in \mathcal{Q}_\mathcal{O}(T)^\times$, with $g(0)\neq 0$, such that
$$\Phi_\rho(\xi)=T^{r_\rho(\xi)}g(T).$$
We call $r_\rho(\xi)$ the order of vanishing of $\xi$ at $\rho$, and denote it by $\ord_{T=0}(\Phi_{\rho}(\xi))$.
\end{defn}

We now relate the evaluation of the characteristic element of a module $M\in\mathfrak{M}_H(G)$ at Artin representations of $G$ to the Akashi series of the twists of $M$ by such representations defined by \eqref{twist}.

\begin{lem}\label{evalxi} Let $M$ be a module in the category $\mathfrak{M}_H(G)$, and let $\xi_M\in K_1(\Lambda(G)_{\Sigma^*})$ denote its characteristic element. Then we have
$$\Phi_\rho(\xi_M)=\mathrm{Ak}(\mathrm{tw}_{\hat{\rho}}(M))\bmod \Lambda_\mathcal{O}(\Gamma)^\times,$$
where $\hat{\rho}$ denotes the contragradient representation of $\rho$, defined by $\hat{\rho}(g)=\rho(g^{-1})^t$ for $g\in G$, where $t$ denotes the transpose matrix.
\end{lem}

\begin{proof} By \cite[Lemma 3.2]{c-non}, the twisted module $\mathrm{tw}_{\hat{\rho}}(M)$ also lies in the category  $\mathfrak{M}_H(G)$. Hence, the Akashi series $\mathrm{Ak}(\mathrm{tw}_{\hat{\rho}}(M))$ is well-defined, and the claimed equality follows from \cite[(57)]{c-non}.
\end{proof}

Given an Artin representation $\rho: G\to GL_n(\mathcal{O})$ of $G$ with coefficients in $\mathcal{O}$, let $F_\rho$ be the Galois extension of $F$ contained in $F_\infty$ such that $\Gal(F_\infty/F_\rho) = \ker \rho$. We then write $G_\rho = \Gal(F_\infty/F_\rho)$. In the case where $\rho=\mathrm{reg}_L$ is the regular representation corresponding to a finite Galois extension $L$ of $F$ contained in $F_\infty$, we simply write $G_L=\Gal(F_\infty/L)$ for $G_{\mathrm{reg}_L}$. We now introduce two assumptions ($\mathbf{G}$) and ($\mathbf{H}$), which will be useful for the computations in this section and in Section~\ref{section9}.
\[
\mathbf{(G)}: \quad H^i(G, A_{p^\infty}(F_\infty)) \text{ is finite for all } i \geq 1.
\]
\[
\mathbf{(H)}: \quad H^i(H, A_{p^\infty}(F_\infty)) \text{ is finite for all } i \geq 0.
\]

\begin{remark}  $(\mathbf{G})$ is clearly satisfied if $A_{p^\infty}(F_\infty)$ is finite. It also holds in the case when $F_\infty=F(A_{p^\infty})$ (see \cite[Section 3]{bv15}). Note that $(\mathbf{H})$ implies $(\mathbf{G})$.
\end{remark}

Similarly, we write ($\mathbf{G_\rho}$) when replacing $G$ with $G_\rho$, and $(\mathbf{G_L})$ in the case where $\rho=\mathrm{reg}_L$ for a finite Galois extension $L$ of $F$ contained in $F_\infty$, with $G_L=\Gal(F_\infty/L)$.

The following proposition provides a criterion for verifying assumption  $(\mathbf{H})$.

\begin{prop}\label{propH} Let $A$ be an abelian variety defined over a global function field $F$ with $\mathrm{char}(F)\neq p$. Suppose that $A_{p^\infty}(F^{\mathrm{cyc}})$ is finite and $H\simeq \BZ_p^r$. Then $(\mathbf{H})$ holds, i.e., $H^i(H,A_{p^\infty}(F_\infty))$ is finite for all $i\geq 0$.
\end{prop}

\begin{proof}
The conclusion follows from the following algebraic observation: If $H\cong \BZ_p^r$, then for a any $\Lambda(H)$-module $M$, the cohomology groups $H_i(H,M)$ are finite for all $i\geq 1$, provided that $H_0(H,M)$ is finite. This result essentially follows from \cite[pp.\,56--57]{Selocal}, although the argument there is somewhat terse. For a detailed proof, we refer the reader to \cite[Proof of Theorem 2.3]{LimKlimit}.
\end{proof}

We also recall, by Remark \ref{remUlmer}, that the finiteness condition on $A_{p^\infty}(F^{\mathrm{cyc}})$ automatically holds when $A$ is an elliptic curve without complex multiplication.

We now prove the following proposition, which highlights the significance of assumption $(\mathbf{G})$.

\begin{prop}\label{prophypG} Let $A$ be an abelian variety defined over a global field of characteristic different from~$p$. Suppose that condition $(\mathbf{G})$ holds. Then, for all $i\geq1$, we have $\ord_{T=0}\left(f_{H^{i}(H,A_{p^\infty}(F_\infty))}\right)=0$.
\end{prop}

\begin{proof} To show that the characteristic series of $H^{i}(H,A_{p^\infty}(F_\infty))$ does not contribute a factor of $T$, it suffices to show that $H^{i}(H,A_{p^\infty}(F_\infty))^\Gamma$ is finite for all $i\geq1$. Since $\mathrm{cd}_p\Gamma=1$, the Hochschild-Serre spectral sequence
$$H^{r}(\Gamma, H^{s}(H, A_{p^\infty}(F_\infty)))\Rightarrow H^{r+s}(G,A_{p^\infty}(F_\infty))$$
degenerates to a short exact sequence
$$0\to H^1(\Gamma, H^{i-1}(H,A_{p^\infty}(F_\infty)))\to H^i(G,A_{p^\infty}(F_\infty))\to H^i(H,A_{p^\infty}(F_\infty))^\Gamma\to 0$$
for every $i\geq 1$. Since $H^i(G,A_{p^\infty}(F_\infty))$ is finite for every $i\geq 1$ by assumption $(\mathbf{G})$, it follows that $H^i(H,A_{p^\infty}(F_\infty))^\Gamma$ is also finite for every $i\geq 1$, as required.
\end{proof}

Now, if $L$ is a finite Galois extension of $F$ contained in $F_\infty$, we write $\mathrm{reg}_L$ for the regular representation of $\Gal(L/F)$, and let $G_L=\Gal(F_\infty/L)$.

\begin{thm}\label{ordreg} Let $A$ be an abelian variety defined over a global function field $F$ with $\mathrm{char}(F)\neq p$, and let $F_\infty/F$ be an admissible $p$-adic Lie extension. Denote by $\xi_A=\xi_{X(A/F_\infty)}$ the characteristic element of $X(A/F_\infty)$. Suppose that the condition $(\mathbf{G_L})$ holds. Then we have $$\ord_{T=0}(\Phi_{\mathrm{reg}_L}(\xi_A))\geq \corank_{\BZ_p}(\Sel(A/L)).$$
Moreover, equality holds if Greenberg's semi-simplicity conjecture \ref{greenberg} is valid for $A/L$.
\end{thm}

\begin{proof} Set $H_L=G_L\cap H$ and $\Gamma_L = \Gal(L^\cyc/L)$. Fix a suitable power of the chosen topological generator of $\Gamma$ so that it is a generator of $\Gamma_L$. With this choice, we have
\[ \BZ_p[[T_L]] \simeq \Lambda(\Gamma_L)\subseteq \Lambda(\Gamma) \simeq \BZ_p[[T]].\]
As discussed in \cite[Subsection 2.2 and Lemma 2.5]{Lim22}, there exists a restriction map
\[ \mathrm{res}: K_1(\Lambda(G)_{\Sigma^*}) \longrightarrow K_1(\Lambda(G_L)_{\Sigma_{G_L}^*})  \]
such that $\mathrm{res}(\xi_A)$ is a characteristic element of $X(A/F_\infty)$ viewed in $\mathfrak{M}_{H_L}(G_L)$.
Furthermore, \cite[Proposition 2.10]{Lim22} gives the identity
\[\ord_{T=0}(\Phi_{\mathrm{reg}_L}(\xi_A)) = \ord_{T_L=0}(\Phi_{\mathrm{reg}_L, G_L}(\xi_A)).\]
Thus, by the observations above, it suffices to treat the case $F=L$. By Lemma \ref{evalxi}, we have
$$\Phi_{\mathrm{reg}_F}(\xi_A)=\mathrm{Ak}(\Sel(A/F_\infty))u(T)$$
for some $u(T)\in \Lambda(T)^\times$.
Under the hypotheses of the theorem, we apply Proposition \ref{prophypG} to conclude that $\ord_{T=0}\left(f_{H^{j}(H,A_{p^\infty}(F_\infty))}\right)=0$ for all $j\geq 1$. Combining this with Theorem \ref{akashithm}, we obtain
$$\ord_{T=0}\left(\mathrm{Ak}\left(\Sel(A/F_\infty)\right)\right)=\ord_{T=0}\left(f_{\Sel(A/F^{\mathrm{cyc}})}\prod_{v\in  S'} f_{J_v(A/F^{\mathrm{cyc}})}\right).$$
By Proposition \ref{controlprop}, we know that the order of vanishing of $f_{\Sel(A/F^{\mathrm{cyc}})}$ at $T=0$ is at least $\corank_{\BZ_p}(\Sel(A/F))$, where equality holds if Conjecture \ref{greenberg} is valid.
It therefore remains to show that the terms $f_{J_v(A/F^{\mathrm{cyc}})}$ for $v\in S'$ do not contribute to the order of vanishing. To see this, it suffices to show that $J_v(A/F^{\mathrm{cyc}})^{\Gamma}$ is finite. Indeed, since $\Gamma_w\simeq \Gamma\simeq \BZ_p$ has $p$-cohomological dimension one, the inflation--restriction sequence yields a surjection
$$H^1(F_v,A_{p^\infty}) \twoheadrightarrow J_v(A/F^{\mathrm{cyc}})^\Gamma.$$
By Tate local duality (cf. \cite[Chapter I, Remark 3.6]{milne}), $H^1(F_v, A_{p^\infty})$ is dual to $A^*_{p^\infty}(F_v)$, which is finite by \cite[Proposition 3.2]{CX}. This completes the proof.
\end{proof}

We now present a corollary that relates the characteristic element to the $L$-function of $A$ over the intermediate subfields of $F_\infty/F$. For a finite extension $L$ of $F$ contained in $F_\infty$, we write $L_S(A/L,s)$ for the $L$-function of $A$ with Euler factors at the primes of $L$ lying above $S$ removed. The following is our corollary.

\begin{cor}\label{corbsd}
Let $A$ be an abelian variety defined over a global function field $F$ with $\mathrm{char}(F)\neq p$, and let $F_\infty/F$ be an admissible $p$-adic Lie extension.
Suppose that for a finite Galois extension $L$ of $F$ contained in $F_\infty$, the following conditions hold:
\begin{enumerate}
  \item The Tate--Shafarevich group $\Sha(A/L)$ is finite.
  \item Greenberg's semi-simplicity conjecture \ref{greenberg} holds for $A/L$.
  \item Condition $(\mathbf{G_L})$ holds.
\end{enumerate}
Then
\[ \ord_{T=0}(\Phi_{\mathrm{reg}_L}(\xi_A))= \ord_{s=1}L_S(A/L,s). \]
\end{cor}

\begin{proof}
Under the finiteness of $\Sha(A/L)$, we have \[\corank_{\BZ_p}(\Sel(A/L)) =\rank_{\mathbb{Z}}A(L).\]
Moreover, the deep results of Artin--Tate \cite{tate}, Milne \cite{milneAT}, and Kato--Trihan \cite{kato-trihan} imply that \[\rank_{\mathbb{Z}}A(L) =\ord_{s=1}L_S(A/L,s),\]
provided that
$\Sha(A/L)$ is finite.
The conclusion then follows from combining these equalities with Theorem \ref{ordreg}.
\end{proof}

It is now natural to ask whether similar results hold for non-regular Artin representations. To explore this, we introduce the following notation. Let $\rho: G\to GL_n(\mathcal{O})$ be an Artin representation of $G$ with coefficients in $\mathcal{O}$. Let $W_\rho$ be a free $\mathcal{O}$-module of rank $n$ realising $\rho$. Denote by $F_\rho$  the Galois extension of $F$ contained in $F_\infty$ such that $\Gal(F_\infty/F_\rho) = \ker \rho$. We  then write $G_\rho = \Gal(F_\infty/F_\rho)$.

\begin{prop}\label{prophypGrho} Let $A$ be an abelian variety defined over a global field of characteristic different from $p$. Suppose that condition $(\mathbf{G_\rho})$ holds. Then, for all $i\geq1$, we have
\[\ord_{T=0}\left(f_{H^{i}(H, \mathrm{tw}_\rho(A_{p^\infty})(F_\infty))}\right)=0.\]
\end{prop}

\begin{proof} Let $H_\rho = G_\rho \cap H$. Consider the spectral sequence
\[ H^r(H/H_\rho, H^s(H_\rho,\mathrm{tw}_\rho(A_{p^\infty})(F_\infty)) \Rightarrow H^{r+s}(H, \mathrm{tw}_\rho(A_{p^\infty})(F_\infty)).  \]
Since $H/H_\rho$ is finite, it suffices to show that the characteristic series of $H^i(H_\rho,\mathrm{tw}_\rho(A_{p^\infty})(F_\infty))$ does not contribute a factor of $T$. By the definition of $H_\rho$, this group acts trivially on $W_\rho$, and so we have
\begin{equation}\label{tensorcoh}
H^i(H_\rho,\mathrm{tw}_\rho(A_{p^\infty})(F_\infty)) = H^i(H_\rho, A_{p^\infty}(F_\infty))\otimes_{\mathcal{O}}W_\rho.
\end{equation}
Since Proposition \ref{prophypG} shows that the characteristic series of $H^i(H_\rho, A_{p^\infty}(F_\infty))$ does not contribute a factor of $T$, it follows from (\ref{tensorcoh}) that the same  holds for $H^i(H_\rho,\mathrm{tw}_\rho(A_{p^\infty})(F_\infty))$. \end{proof}

Given any infinite separable extension $K$ of $F$ contained in $F_S$, we define the twisted Selmer group $\Sel(\mathrm{tw}_\rho(A)/K)$ by
$$\Sel(\mathrm{tw}_{\rho}(A)/K)=\ker\left(H^1(G_S(K),\mathrm{tw}_\rho(A_{p^\infty}))\longrightarrow \bigoplus_{v\in S}J_v(\mathrm{tw}_\rho(A)/K)\right),$$
where for each $v\in S$,
$$J_v(\mathrm{tw}_\rho(A)/K):=\varinjlim_{L}\bigoplus_{w\mid v}H^1(L_w,\mathrm{tw}_\rho(A_{p^\infty})).$$
Here, the direct limit is taken over all finite extensions
$L$ of $F$ contained in $K$, and $w$ runs over all primes of $L$ lying above $v$. The transition maps in the limit are given by restriction.
We denote by $X(\mathrm{tw}_\rho(A)/K)$ the Pontryagin dual of $\Sel(\mathrm{tw}_\rho(A)/K)$.

\begin{lem} \label{twistsur}
We have the following short exact sequence of cofinitely generated $\Lambda_{\mathcal{O}}(H)$-modules
$$0 \longrightarrow \Sel(\mathrm{tw}_{\rho}(A)/F_\infty) \longrightarrow H^1(G_S(F_\infty),\mathrm{tw}_\rho(A_{p^\infty}))\longrightarrow \bigoplus_{v\in S}J_v(\mathrm{tw}_\rho(A)/F_\infty)\longrightarrow 0.$$
\end{lem}

\begin{proof}
By the definition of $\rho$, the group $G_S(F_\infty)$ acts trivially on $W_\rho$. Hence, we have $\Sel(\mathrm{tw}_{\rho}(A)/F_\infty) = \Sel(A/F_\infty)\otimes_{\mathcal{O}}W_\rho$.
Since $\Sel(A/F_\infty)$ is cofinitely generated over $\Lambda(H)$ by Lemma \ref{corankglo}, it follows that $\Sel(\mathrm{tw}_{\rho}(A)/F_\infty)$ is cofinitely generated over $\Lambda_{\mathcal{O}}(H)$. Similarly, we know that
\[H^1(G_S(F_\infty),\mathrm{tw}_\rho(A_{p^\infty}))\quad \textrm{and}\quad  \bigoplus_{v\in S}J_v(\mathrm{tw}_\rho(A)/F_\infty)\]
are
cofinitely generated over $\Lambda_{\mathcal{O}}(H)$. Since $W_\rho$ is a free $\mathcal{O}$-module, the short exact sequence of the lemma then follows from Theorem \ref{thmsurj} by tensoring the corresponding short exact sequence in that theorem with ~$W_\rho$.
\end{proof}

By an argument similar to that in \cite[Proposition 2.5]{coates-sujatha12}, the restriction map
\[ \Sel(\mathrm{tw}_{\rho}(A)/F^\cyc) \rightarrow \Sel(\mathrm{tw}_{\rho}(A)/F_\infty)^H \]
has kernel which is cofinitely generated over $\mathcal{O}$. It then follows from this fact and the preceding lemma that $\Sel(\mathrm{tw}_{\rho}(A)/F^\cyc)$ is cofinitely generated over $\mathcal{O}$. Therefore, we can apply the structure theory for $\Lambda_{\mathcal{O}}(\Gamma)$-modules to obtain the following pseudo-isomorphism

\begin{equation}\label{eqgr2}X(\mathrm{tw}_\rho(A)/F^{\mathrm{cyc}})\sim  \bigoplus_{j=1}^t \Lambda_{\mathcal{O}}(\Gamma)/(f_j^{\beta_j})
\end{equation}
where each $f_j$ is an irreducible distinguished polynomial.

The following is Greenberg's semi-simplicity conjecture for Artin twists over global function fields, which serves as the function field analogue to the conjecture in \cite[Conjecture 6.4]{Lim22}.

\begin{conj}[Greenberg's semi-simplicity conjecture for $\mathrm{tw}_\rho(A)/F$]\label{greenberg2} Let $A$ be an abelian variety defined over a global function field $F$ with ${\rm char}(F)\neq p$. Then the action of $\Gamma$ on $X(\mathrm{tw}_\rho(A)/F^{\mathrm{cyc}})$ is completely reducible, that is, $\beta_j=1$ for all $j$ in \eqref{eqgr2}.
\end{conj}

We are now in position to present the following result on the order of vanishing of characteristic elements evaluated at Artin representations.

\begin{thm}\label{ordrho}  Let $A$ be an abelian variety defined over a global function field $F$ with $\mathrm{char}(F)\neq p$, and let $F_\infty/F$ be an admissible $p$-adic Lie extension. Let $\rho$ be an irreducible Artin representation of $G=\Gal(F_\infty/F)$, and let $\xi_A$ denote a characteristic element of $X(A/F_\infty)$. Suppose that condition $(\mathbf{G_\rho})$ holds. Then
$$\ord_{T=0}(\Phi_{\rho}(\xi_A))\geq \corank_{\mathcal{O}}(\Sel(\mathrm{tw}_\rho(A)/F)).$$
Moreover, equality holds if Greenberg's semi-simplicity conjecture \ref{greenberg2} holds for $\mathrm{tw}_\rho (A)/F$.
\end{thm}

\begin{proof}
The arguments used in the proof of Theorem \ref{ordreg} carry over to the setting of the Artin twists, using Proposition \ref{prophypGrho} and Lemma \ref{twistsur}. In the proof, we also make use of the identity
$X(\mathrm{tw}_\rho(A)/F_\infty) = \mathrm{tw}_{\hat{\rho}}X(A/F_\infty)$, as shown in \cite[Lemma 3.4]{c-fks}. \end{proof}

\section{The Generalised Euler characteristics}\label{section8}

In this section, we introduce the definition of the generalised Euler characteristic. The idea of defining a finite Euler characteristic $\chi(G,Y)$ for a discrete $p$-primary $G$-module $Y$, even when the cohomology groups $H^i(G,Y)$ are infinite, was first proposed by Coates--Schneider--Sujatha in the form of the truncated Euler characteristic \cite{coates-sujatha-schneider}.  Our presentation here follows that of Zerbes \cite{zerbes2}.

\medskip

Let $W$ be a discrete $p$-primary $\Gamma$-module. Then we have
\[H^0(\Gamma,W)=W^\Gamma\quad H^1(\Gamma,W)=W_\Gamma,\]
where $W_\Gamma$  is the maximal quotient of $W$ on which $\Gamma$ acts trivially.
Hence, we have a map
\[\phi_W: H^0(\Gamma,W)\to H^1(\Gamma,W),\quad f\mapsto \textrm{ residue class of } f.\]
Now, let $Y$ be a discrete $p$-primary $G$-module, we define
\[d^0: H^0(G,Y)=H^0(\Gamma,Y^H)\longrightarrow H^1(\Gamma,Y^H)\hookrightarrow H^1(G,Y),\]
where the middle map is $\phi_{Y^H}$, and the last map is the inflation map. Similarly, for each $i\geq1$, define
\[d^i: H^i(G,Y)\twoheadrightarrow H^0(\Gamma, H^i(H,Y))\longrightarrow H^1(\Gamma,H^i(H,Y))\hookrightarrow H^{i+1}(G,Y),\]
where the first map is the restriction map, the middle map is $\phi_{H^i(H,Y)}$. Note that the restriction map is surjective since
${\rm cd}_p(\Gamma)=1$. We define $d^{-1}$ to be the zero map.

\bigskip

All the properties for the morphisms involved in the definition of the complex $(H^i(G,Y),d^i)$ will become clear in the proof of the following proposition.

\medskip

\begin{prop}\label{3-1}
The sequence $(H^\bullet(G,Y),d^\bullet)$ forms a cochain complex.
\end{prop}

\begin{proof}
We consider the Hochschild--Serre spectral sequence
\[H^r(\Gamma,H^s(H,Y))\Rightarrow H^{r+s}(G,Y).\]
Since ${\rm cd}_p(\Gamma)=1$, it follows that $H^r(\Gamma,H^s(H,Y))=0$ unless $r=0$ or $1$.
Thus, for all $s\geq1$,  we obtain the short exact sequence
\begin{equation}\label{3-f}
 0\to H^1(\Gamma,H^{s-1}(H,Y)) \to H^s(G,Y)\to H^0(\Gamma,H^s(H,Y)) \to 0.
\end{equation}

Now, observe that the differential $d^s: H^s(G,Y)\to H^{s+1}(G,Y)$ is given by the composition
\[H^s(G,Y)\twoheadrightarrow H^0(\Gamma,H^s(H,Y)) \stackrel{\phi_{H^s(H,X)}}{\longrightarrow}H^1(\Gamma,H^{s}(H,Y))\hookrightarrow H^{s+1}(G,Y),\]
where the surjection and injection come directly from the short exact sequence \eqref{3-f}.
Therefore, $d^{s+1}\circ d^s$ factors through
\[H^1(\Gamma,H^{s}(H,Y))\to H^{s+1}(G,Y) \to H^0(\Gamma,H^{s+1}(H,Y)),\]
which,  by exactness of \eqref{3-f}, is the zero map. Hence, we conclude that $d^{s+1}\circ d^s=0$, as required.
\end{proof}

\medskip

\begin{defn}\label{3-2} Let $\fH^\bullet$ denote the cohomology of the complex constructed in Proposition \ref{3-1}. If $\fH^i$ is finite for all $i\geq0$, we say that the $G$-module $Y$ has a finite generalised $G$-Euler characteristic $\chi(G,Y)$, defined by
\[\chi(G,Y)=\prod_{i\geq0}\#(\fH^i)^{(-1)^i}.\]
\end{defn}

\begin{lem}\label{3-3} The generalised $G$-Euler characteristic satisfies the following properties:
\noindent
\begin{enumerate}
  \item [(1)]If the groups $H^i(G,Y)$ are finite for all $i$, then the generalised $G$-Euler characteristic of $Y$ agrees with the usual $G$-Euler characteristic.
  \item [(2)]Suppose $H^0(G,Y)$ and $H^1(G,Y)$ are infinite, and $H^i(G,Y)$ is finite for all $i\geq2$. Then $Y$ has a finite generalised $G$-Euler characteristic if and only if the map $d^0$ has a finite kernel and cokernel, in which  case  \begin{equation}\label{3-f-1}
    \chi(G,Y)=\frac{\#(\ker(d^0))}{\#(\coker(d^0))}\times\prod_{i\geq2}\#\left(H^i(G,Y)\right)^{(-1)^i}.
  \end{equation}
\end{enumerate}
\end{lem}

\begin{proof}
This result is essentially contained in the work of Zerbes \cite{zerbes2}, though not explicitly stated in this form. For the reader's convenience, we provide a full proof.
For part (1), it suffices to show that for a cochain complex $(\fX^\bullet ,\delta^\bullet)$ with each $\fX^i$ finite for  $i\geq 0$, we have
\[\prod^\infty_{i=0}\#(H^i(\fX^\cdot))^{(-1)^i}=\prod^{\infty}_{i=0}\#(\fX^i)^{(-1)^i}.\]
To verify this, consider the following two short exact sequences
\[0\to Z^i\to \fX^i \to B^{i+1}\to0,\]
\[0\to B^i\to Z^i \to H^i(\fX^\cdot)\to 0.\]
Since all modules involved are finite, we have
\[\#(\fX^i)=\#(Z^i)\cdot\#(B^{i+1}).\]
Because $\#(B^0)=1$, it follows that
\[\prod_{i\geq0}\#(\fX^i)^{(-1)^i}=\prod_{i\geq0}(\#(Z^i)
\cdot\#(B^{i+1}))^{(-1)^i}=
\prod_{i\geq0}\#(H^i(\fX^\cdot))^{(-1)^i}.\]
This proves  part (1). For part (2), first note that $\ker(d^0)=\fH^0$. Thus, if $d^0$ has a finite kernel and cokernel, then $\fH^0$ is finite. Consider
the exact sequence
\[0\to \fH^1 \to \coker d^0 \to H^2(G,Y), \]
where the third map is induced by $d^1$. Since $\coker d^0$ and $H^2(G,Y)$ are finite, it follows that $\fH^1$ is also finite. Given that $\fH^i$ is finite for all $i\geq2$ by assumption, thus $Y$ has a finite generalised $G$-Euler characteristic. Conversely, suppose that $Y$ has a finite generalised $G$-Euler characteristic. Since $\ker(d^0)=\fH^0$ and $\fH^0$ is finite, $d^0$ must have finite kernel. Because $H^2(G,Y)$ is finite, so is $\ker(d^2)$. Now consider the exact sequence
\[0\to \fH^1 \to \coker d^0 \to \ker d^2 \to \fH^2 \to 0.\]
It follows that $\coker (d^0)$ is finite as well. To establish the formula \eqref{3-f-1}, note that the complex $(H^\bullet(G,Y),d^\bullet)$ induces the following cochain complex
\[0\to\coker(d^0)\stackrel{\bar{d}^1}{\longrightarrow} H^2(G,Y)\stackrel{d^2}{\longrightarrow}\cdots.\]
All terms in this complex are finite, and $\ker(\bar{d}^1)\simeq\fH_1$. By part (1), we obtain \begin{equation}\label{3-f-2}
\prod^\infty_{i=1}\#(\fH^i)^{(-1)^{i+1}}=\#(\coker (d^0))\times \prod^\infty_{i=2}\#\left(H^i(G,Y)\right)^{(-1)^{i+1}}.
\end{equation}
Since $\#(\fH^0)=\#(\ker(d^0))$, dividing both sides of  \eqref{3-f-2} by $\#(\fH^0)$ gives
\[\chi(G,Y)=\frac{\#(\ker(d^0))}{\#(\coker(d^0))}\times\prod_{i\geq2}\#\left(H^i(G,Y)\right)^{(-1)^i},\]
which is exactly the expression in \eqref{3-f-1}.
\end{proof}

\bigskip

Let $| \cdot |_p$ denote the $p$-adic valuation of $\overline{\BQ}$, normalised so that $|p|_p = p^{-1}$.
We now record the following important result, which relates the generalised Euler characteristic to the Akashi series.

\begin{prop}\label{3-5}
Let $Y$ be a discrete $p$-primary $G$-module such that $Y^\vee\in\fM_H(G)$. If $Y$ has a finite generalised $G$-Euler characteristic, then the leading term of its Akashi series $\mathrm{Ak}(Y)$ is given by $\alpha_Y T^k$, where $k$ is  the alternating sum
\[k=\sum_{i\geq0}(-1)^i {\rm corank}_{\BZ_p}H^i(H,Y)^\Gamma\]
and
\[\chi(G,Y)=|\alpha_Y|^{-1}_p.\]
\end{prop}
\begin{proof}
See \cite[Proposition 2.10]{zerbes2}.
\end{proof}

We remark that, in general, the generalised $G$-Euler characteristic is not multiplicative in short exact sequences (see \cite[the remark following Lemma 2.13]{zerbes2}). However, the following lemma suffices for most of our purposes.

\begin{lem}\label{fingec}
Let $$0\to B_1 \to B_2 \to B_3 \to B_4 \to 0$$
be an exact sequence of discrete $G$-modules such that $B_1$ is finite and $B_4$ has a finite $G$-Euler characteristic. Then $B_2$ has a finite generalised $G$-Euler characteristic if and only if $B_3$ does.
\end{lem}

\begin{proof}
Set $C=\ker(B_3\to B_4)$. By analysing the following two short exact sequences separately:
\[0\to B_1 \to B_2 \to C \to 0,\]
and
\[0\to C \to B_3 \to B_4\to 0,\]
we may reduce to the case where either $B_1=0$ or $B_4=0$. We  prove the lemma in the case $B_1=0$; the case of $B_4=0$ is analogous. Assume, therefore, that we have a short exact sequence of $G$-modules
\[0\to B_2\to B_3\to B_4\to 0,\]
with $B_4$ having finite $G$-Euler characteristics. In particular, $H^i(G,B_4)$ is finite for all $i\geq0$. Consider the long exact sequence in $G$-cohomology arising from the short exact sequence above. It yields morphisms
\[\psi_i:H^i(G,B_2)\to H^i(G, B_3)\]
with finite kernel and cokernel for all $i\geq0$.
By the naturality of the Hochschild--Serre spectral sequence and the definition of $d^i$, the maps $\psi_i$ are cochain maps between the complexes $(H^\bullet(G,B_2),d^\bullet_{B_2})$ and $(H^\bullet(G,B_3),d^\bullet_{B_3})$. Since the kernel and cokernel of these cochain maps are finite, standard homological algebra argument implies that the induced maps on the cohomology of these complexes also have a finite kernel and cokernel.
Therefore, $B_2$ has a finite generalised Euler characteristic if and only if $B_3$ does.
\end{proof}

\begin{remark} One can even show that \[\chi(G,B_2)\chi(G, B_4)=\chi(G,B_1)\chi(G,B_3)\] in the preceding lemma. As this latter calculation is not required for the paper, we do not provide a detailed proof, leaving it to the interested readers.
\end{remark}

\section{Computations of the Generalised Euler Characteristics}\label{section9}

To compute the generalised Euler characteristics, we impose the following assumption $(\mathbf{H})$ on the abelian variety $A$ defined over a function field $F$, which we now recall:

\[
\mathbf{(H)}: \quad H^i(H, A_{p^\infty}(F_\infty)) \text{ is finite for all } i \geq 0.
\]

\begin{thm}\label{GECthm}
Let $A$ be an abelian variety defined over a global field $F$ with $\mathrm{char} F\neq p$, and let $F_\infty$ be an admissible $p$-adic Lie extension of $F$. Suppose that assumption $(\mathbf{H})$ holds. Then $\Sel(A/F_\infty)$ has a finite generalised $G$-Euler characteristic if and only if $\Sel(A/F^{\rm cyc})$ has a finite generalised $\Gamma$-Euler characteristic.  In that case, we have the following equality
\[\chi(G, \Sel(A/F_\infty)) = \chi(\Gamma, \Sel(A/F^\cyc)) \prod_{v\in S'}\frac{\#  A^*_{p^\infty}(F_v)}{\# H^1(\Gamma_w, A_{p^\infty}(F^\cyc_w))},\]
where, for each $v\in S'$, $w$ is a fixed prime of $F^\cyc$ lying above $v$, and $\Gamma_w$ denotes the decomposition subgroup of $\Gamma$ at $w$.
\end{thm}

\begin{proof}
Consider the following fundamental commutative diagram
\begin{footnotesize}
\begin{equation}\label{fundseq}
 \xymatrix{
  0  \ar[r]^{} & \Sel(A/F^{\rm cyc}) \ar[d]_{\alpha} \ar[r]^{} & H^1(G_S(F^{\rm cyc}),A_{p^\infty}) \ar[d]_{\beta} \ar[r]^{} & \bigoplus_{v\in S}J_v(A/F^{\rm cyc})\ar[d]^{\gamma} \ar[r]^{} & 0 & \\
  0 \ar[r]^{} & \Sel(A/F_\infty)^H \ar[r]^{} & H^1(G_S(F_\infty),A_{p^\infty})^H \ar[r]^{\phi} & \bigoplus_{v\in S}J_v(A/F_\infty)^H \ar[r] &  H^1(H, \Sel(A/F_\infty)) \ar[r] & \cdots.  &}
\end{equation}
\end{footnotesize}
The rows are exact. From the proof of Lemma \ref{AKlocal}, we know that the $p$-cohomological dimension of the decomposition group of $H$ at every prime of $F_\infty$ is at most one. It follows that the map $\gamma$ is surjective. By applying the Snake lemma, it follows that the map $\phi$ is also surjective. Applying the snake lemma to the diagram \eqref{fundseq}, we obtain the following exact sequence
\begin{equation}\label{snakeseq}
0\to \ker \alpha \to H^1(H,A_{p^\infty}(F_\infty)) \to \ker \gamma \to \coker \alpha \to H^2(H,A_{p^\infty}(F_\infty)) \to 0.
\end{equation}
From Lemma \ref{AKlocal}, we have $H^i(H, J_v(A/F_\infty))=0$ for every $i\geq 1$. Noting that $\phi$ is surjective and using the long exact sequence in $H$--cohomology associated to the short exact sequence \eqref{thmsurj-f001}, we obtain an isomorphism
\[H^i(H,\Sel(A/F_\infty))\simeq H^i(H,H^1(G_S(F_\infty),A_{p^\infty}))\qquad \textrm{for all $i\geq1$.}\]
Consider the Hochschild--Serre spectral sequence
\[H^p(H,H^q(G_S(F_\infty),A_{p^\infty}))\Rightarrow H^{p+q}(G_S(F^{\rm cyc}),A_{p^\infty}).\]
From Lemma \ref{cdpcyc}, for every $i\geq1$, we have
\[H^i(H,H^1(G_S(F_\infty),A_{p^\infty}))\simeq H^{i+2}(H,A_{p^\infty}(F_\infty)).\]
Thus, for all $i\geq1$,
\[H^i(H,\Sel(A/F_\infty))\simeq H^{i+2}(H,A_{p^\infty}(F_\infty)).\]
By assumption $(\mathbf{H})$, $H^i(H,\Sel(A/F_\infty))$ is finite for all $i\geq 1$. Consider now the following exact sequence
\[0\to H^1(\Gamma,H^i(H,\Sel(A/F_\infty)))\to H^{i+1}(G,\Sel(A/F_{\infty}))\to H^0(\Gamma,H^{i+1}(H,\Sel(A/F_\infty)))\to 0.\]
From Lemma \ref{3-3}, $\Sel(A/F_\infty)$ has a finite generalised $G$-Euler characteristic if and only if the map
\[ d^0 : H^0(G, \Sel(A/F_\infty)) = H^0(\Gamma, \Sel(A/F_\infty)^H) \stackrel{g}{\to} H^1(\Gamma, \Sel(A/F_\infty)^H) \stackrel{h}{\hookrightarrow} H^1(G, \Sel(A/F_\infty))  \]
has a finite kernel and cokernel. Since $\coker h = H^1(H,\Sel(A/F_\infty))^\Gamma\simeq H^{3}(H,A_{p^\infty}(F_\infty))^\Gamma$ is also finite, it follows that $d^0$ has a finite kernel and cokernel if and only if $g$ does. The latter is precisely equivalent to saying that $\Sel(A/F_\infty)^H$ has a finite generalised $\Gamma$-Euler characteristic.

On the other hand, from diagram (\ref{fundseq}), we have the following exact sequence
\begin{equation}\label{fundseq2}
  0\to \ker \alpha \to \Sel(A/F^{\rm cyc}) \to \Sel(A/F_\infty)^H \to \coker \alpha \to 0.
\end{equation}
Since $\ker \alpha\subseteq H^1(H,A_{p^\infty}(F_\infty))$, we know $\ker \alpha$ is finite. We claim that $\coker \alpha$ has a finite $\Gamma$-Euler characteristic. Assuming the claim, it follows from Lemma \ref{fingec} and (\ref{fundseq2}) that $\Sel(A/F_\infty)^H$ has a finite generalised $\Gamma$-Euler characteristic if and only if $\Sel(A/F^\cyc)$ does. This establishes the first assertion of the theorem. To prove the claim, note that by \eqref{snakeseq}, it suffices to show that
$\ker \gamma$ has a finite $\Gamma$-Euler characteristic. As observed in the proof of Theorem  \ref{ordreg}, there is a surjection
$$H^1(F_v,A_{p^\infty}) \twoheadrightarrow J_v(A/F^{\mathrm{cyc}})^\Gamma,$$
and since $H^1(F_v, A_{p^\infty})$ is finite, so is $J_v(A/F^{\mathrm{cyc}})^\Gamma$. Thus, we see that $({\rm ker}\gamma)^\Gamma$ is also finite. By \cite[Proposition A.1.7]{coates-sujatha}, it follows that ${\rm ker}(\gamma)$ has a finite $\Gamma$-Euler characteristic.

Now suppose that both generalised Euler characteristics in question are well-defined. Then by Theorem \ref{akashithm} and Proposition \ref{3-5}, we obtain
\begin{equation}\label{AkaEul}
\chi(G, \Sel(A/F_\infty)) = \chi(\Gamma, \Sel(A/F^\cyc)) \prod_{v\in S'}\chi(\Gamma,  J_v(A/F^{\mathrm{cyc}})).
\end{equation}
Thus, the final assertion of the theorem will follow once we show that
\begin{equation}\label{localGammaEul}
  \chi(\Gamma,  J_v(A/F^{\mathrm{cyc}})) = \frac{\#  A^*_{p^\infty}(F_v)}{\# H^1(\Gamma_w, A_{p^\infty}(F^{\cyc}_w))}
\end{equation}
 for each $v\in S'$. The restriction map gives the following short exact sequence
\[ 0 \to H^1(\Gamma_w, A_{p^\infty}(F^{\cyc}_w)) \to H^1(F_v, A_{p^\infty}) \to J_v(A/F^{\mathrm{cyc}})^\Gamma \to 0,\]
where, using Shapiro's lemma, we identify
\[J_v(A/F^{\rm cyc})^\Gamma\simeq H^0(\Gamma_w,H^1(F^{\rm cyc}_w,A_{p^\infty})).\]
By Tate local duality, $H^1(F_v, A_{p^\infty})$ is dual to $ A^*_{p^\infty}(F_v)$. Therefore,
\begin{equation}\label{localGammaEul1}\# J_v(A/F^{\mathrm{cyc}})^\Gamma = \frac{\#  A^*_{p^\infty}(F_v)}{\# H^1(\Gamma_w, A_{p^\infty}(F^{\cyc}_w))}.
\end{equation}
On the other hand, by Shapiro's lemma, we have
\begin{equation}\label{localGammaEul2} H^1\left(\Gamma, J_v(A/F^{\mathrm{cyc}})\right) \simeq H^1(\Gamma_w, H^1(F^\cyc_w, A_{p^\infty})).
\end{equation}
Now consider the following spectral sequence
$$H^r(\Gamma_w, H^s(F^\cyc_w, A_{p^\infty}))\Rightarrow H^{r+s}(F_v, A_{p^\infty}).$$
Since $\mathrm{cd}_p ~\Gamma_w =1$ and $H^2(F^\cyc_w, A_{p^\infty}) =0$ (cf. \cite[Theorem 7.1.8(i)]{nsw}), the spectral sequence degenerates, yielding an isomorphism
\begin{equation}\label{localGammaEul3} H^1(\Gamma_w, H^1(F^\cyc_w, A_{p^\infty})) \cong H^2(F_v, A_{p^\infty}).
\end{equation}
By Tate local duality, $H^2(F_v, A_{p^\infty})=0$ (see \cite[Lemma 1.12]{coates-sujatha}). Therefore,  combining (\ref{localGammaEul1}), (\ref{localGammaEul2}) and (\ref{localGammaEul3}), we obtain the identity (\ref{localGammaEul}), as required. This completes the proof of the theorem.
\end{proof}

\begin{remark} When $A=E$ is an elliptic curve, it follows from well-known properties of elliptic curves that
\[\chi(\Gamma,J_v(E/F^{\rm cyc}))=|L_v(E,1)|_p\]
where $L_v(E,1)$ denotes the Euler factor at $v$ of the complex $L$-function of $E$ over $F$. This recovers a formula of Zerbes in the number field setting and a corresponding result of Sechi in the function field setting using standard Euler characteristics.
\end{remark}

In \cite[Theorem 3.2.6]{LLTT}, Lai--Longhi--Tan--Trihan computed the generalised $\Gamma$-Euler characteristic
$$\chi(\Gamma, \Sel(A/F^\cyc))$$
in terms of the quantities appearing in the Birch--Swinnerton-Dyer conjecture. Combining their result with ours, we deduce that the generalised $G$-Euler characteristic $\chi(G, \Sel(A/F_\infty))$ can likewise be expressed in terms of invariants arising in the Birch--Swinnerton-Dyer conjecture. We conclude with the following corollary, which illustrates how the characteristic element fits into this framework.

\begin{cor} \label{GenEul coro}
Retain the setting of Theorem 9.1. Suppose further that Greenberg's semi-simplicity conjecture \ref{greenberg} holds for $A/F$. Set $s_A=\mathrm{corank}_{\BZ_p} \Sel(A/F)$. Then
\[ \frac{1}{T^{s_A}}\Phi_{\mathrm{reg}_F}(\xi_A) \Big|_{T=0} = \chi(\Gamma, \Sel(A/F^\cyc)) \prod_{v\in S'}\frac{\#  A^*_{p^\infty}(F_v)}{\# H^1(\Gamma_w, A_{p^\infty}(F^\cyc_w))},\]
where, for each $v\in S'$, $w$ is a fixed prime of $F^\cyc$ lying above $v$, and $\Gamma_w$ is the corresponding decomposition subgroup.
\end{cor}

\begin{proof}
This follows from a combination of Theorems \ref{ordreg} and \ref{GECthm}.
\end{proof}

\subsection*{Acknowledgements}

The authors thank Ashay Burungale, Pierre Colmez, Mahesh Kakde, Antonio Lei, Katharina M\"{u}ller, Anwesh Ray, Jishnu Ray, Gianluigi Sechi, Ramdorai Sujatha, Fabien Trihan, Maria Valentino and Sarah Zerbes for their interest and comments. They are very grateful to the anonymous referee for valuable suggestions. This paper arose from the conference ``Special Values of $L$-functions'' held at Tsinghua Sanya International Mathematics Forum, and the authors thank Ashay Burungale, Emmanuel Lecouturier and Xin Wan for organising the conference that enabled the authors' collaboration. Yukako Kezuka also thanks Gianluigi Sechi for sharing his PhD thesis, which inspired their collaboration. Yukako Kezuka was partially supported by the ANR project Coloss, the European Union's Horizon 2020 research and innovation programme under the Marie Sklodowska-Curie grant agreement No.\,101026826, and JSPS KAKENHI Grant JP25K17227.\\

\noindent
\textbf{Conflict of interest.} On behalf of all authors, the corresponding author states that there is no conflict of interest.

\end{document}